\documentclass[11pt,a4paper]{amsart}
\usepackage{amssymb}
\usepackage{latexsym}
\usepackage{exscale}
\usepackage{amsfonts}
\usepackage{graphicx}
\usepackage{mathrsfs}
\usepackage{amsmath,amscd,amsthm}
\usepackage{bbm,pifont}
\usepackage{enumerate}
\usepackage{color}
\usepackage{amssymb}
\usepackage{latexsym}
\usepackage{exscale}
\usepackage{color}

\allowdisplaybreaks

\DeclareMathAlphabet\mathcalbf{OMS}{cmsy}{b}{n}
\DeclareMathAlphabet\EuScript{U}{eus}{m}{n}
\DeclareMathAlphabet\EuScriptBold{U}{eus}{b}{n}

\numberwithin{equation}{section}

\newcommand{\supp}{\operatorname{supp}}

\def\supp{\mathop{\rm supp}}

\def\bint{{\ifinner\rlap{\bf\kern.30em--}
\int\else\rlap{\bf\kern.35em--}\int\fi}\ignorespaces}

\def\sbint{{\ifinner\rlap{\bf\kern.32em--}
\hspace{0.078cm}\int\else\rlap{\bf\kern.45em--}\int\fi}\ignorespaces}

\def\rr{{\mathbb R}}

\def\nn{{\mathbb N}}

\def\cc{\mathscr H^{n-1}}

\def\cf{{\mathcal F}}

\def\cm{{\mathcal M}}

\def\fz{\infty}
\def\az{\alpha}
\def\supp{{\mathop\mathrm{\,supp\,}}}

\def\loc{{\mathop\mathrm{\,loc\,}}}

\def\dz{\delta}

\def\ez{\epsilon}

\def\wz{\widetilde}

\def\ls{\lesssim}
\def\gs{\gtrsim}

\def\cmoc{{{\rm VMO}(\cc)}}
\def\bmoc{{{\rm BMO}(\cc)}}

\def\r{\right}
\def\lf{\left}

\def\beeqn{\begin{equation}}
\def\eneqn{\end{equation}}
\def\beeqns{\begin{equation*}}
\def\eneqns{\end{equation*}}
\def\beeqa{\begin{eqnarray}}
\def\eneqa{\begin{eqnarray}}
\def\beeqas{\begin{eqnarray*}}
\def\eneqas{\begin{eqnarray*}}
\def\besp{\begin{split}}
\def\ensp{\begin{split}}
\def\noz{\nonumber}

\newtheorem{thm}{Theorem}[section]
\newtheorem{lem}[thm]{Lemma}
\newtheorem{rem}[thm]{Remark}

\theoremstyle{definition}

\numberwithin{equation}{section}

\begin{document}
\allowdisplaybreaks

\title[]
{Cauchy--Szeg\"{o} Commutators on Weighted Morrey Spaces}

\author{ Zunwei Fu, Ruming Gong, Elodie Pozzi and Qingyan Wu}

\address{Zunwei Fu, Department of Mathematics\\
         Linyi University\\
         Shandong, 276005, China
         }
\email{fuzunwei@lyu.edu.cn}

\address{Ruming Gong, Department of Mathematics\\
         Guangzhou University\\
         Guangzhou, 510006 , China
         }
\email{gongruming@gzhu.edu.cn}

\address{Elodie Pozzi, Department of Mathematics and Statistics\\
         Saint Louis University\\
         220 N. Grand Blvd, 63103 St Louis MO, USA
         }
\email{elodie.pozzi@slu.edu}

\address{Qingyan Wu, Department of Mathematics\\
         Linyi University\\
         Shandong, 276005, China
         }
\email{wuqingyan@lyu.edu.cn}

%

\date{\today}

\begin{abstract}
In the setting of quaternionic Heisenberg group $\mathscr H^{n-1}$, we characterize the boundedness and compactness of commutator $[b,\mathcal C]$ for the Cauchy--Szeg\"o operator $\mathcal C$ on the 
weighted Morrey space $L_w^{p,\,\kappa}(\mathscr H^{n-1})$ with $p\in(1, \infty)$, $\kappa\in(0, 1)$ and
$w\in A_p(\mathscr H^{n-1}).$ More precisely, we prove that $[b,\mathcal C]$ is bounded on $L_w^{p,\,\kappa}(\mathscr H^{n-1})$ if and only if $b\in {\rm BMO}(\cc)$. 
And $[b,\mathcal C]$ is compact on $L_w^{p,\,\kappa}(\mathscr H^{n-1})$ if and only if $b\in {\rm VMO}(\cc)$.

\end{abstract}

\maketitle


\section{Introduction and statement of main results}
\setcounter{equation}{0}

Since 1980s, it is an active direction to develop a theory for quaternionic regular functions of several variables instead of holomophic functions on $\mathbb C^n$. 
 Let $\mathbb H$ be the algebra of quaternion numbers and let $\operatorname{Re}x$ and $\operatorname{Im}x$ denote the real part and imaginary part of $x$ respectively. Then $\operatorname{Re}x=x_1$ and $\operatorname{Im}x=x_2{\bf i}+x_3{\bf j}+x_4{\bf k}$. The $n$-dimensional quaternionic space $\mathbb H^n$ is the collection of $n$-tuples $(q_1, \cdots, q_n)$.
For $l$-th coordinate of a point $q=(q_1, \cdots, q_n)\in\mathbb H^n$ we write $q_l=x_{4l-3}+x_{4l-2}{\bf i}+x_{4l-1}{\bf j}+x_{4l}{\bf k}$.
An $\mathbb H$-valued function $f: \Omega\to\mathbb H$ over a domain $\Omega\subset\mathbb H^n$ is called regular if
$\bar\partial_{q_l} f (q)=0,$ $l=1,\cdots, n$, where
$$\bar\partial_l={{\partial\over\partial x_{4l-3}}}+{\bf i} {\partial\over\partial x_{4l-2}}+{\bf j} {\partial\over\partial x_{4l-1}}+{\bf k} {\partial\over\partial x_{4l}}.$$
So far several fundamental results have been established for the quaternionic counterparts, e.g., Hartogs phenomenon, $k$-Cauchy--Fueter complexes, quaternionic Monge-Ampere equations, etc.  (see for example \cite{alesker1,CSSS,wan-wang,wang19} and the references therein).
Since the quaternions $\mathbb H$ is non-commutative, the behavior of quaternionic regular functions is quite different from holomorphic functions, e.g. the product of two such functions is not regular in general. Hence, proofs and even statements of results are completely different from the standard setting of complex variables.

It is natural to consider the Hardy space of regular functions
over a bounded domain in $\mathbb H^n$, in particular, over the unit ball. By quaternionic Cayley transformation, it is equivalent to consider the Hardy space over the Siegel upper half domain
%
%
%
%
%
$$\mathcal U_n:=\left\{q=(q_1,\cdots, q_n)=(q_1,q')\in\mathbb H^n\mid \operatorname{Re} q_1>|q'|^2\right\},$$
where $q'=(q_2,\cdots,q_n)\in\mathbb H^{n-1}$, whose boundary
$\partial \mathcal U_n:=\{(q_1,q')\in\mathbb H^n\mid \operatorname{Re}q_1=|q'|^2\}$
 is a quadratic hypersurface, which can be identified with the quaternionic Heisenberg group
 $\mathscr H^{n-1}$.

 For any function $F:\mathcal U_n\longrightarrow\mathbb H$, we write $F_\varepsilon$ for its ``vertical translate", where the vertical direction is given by the positive direction of $\operatorname{Re}q_1: F_\varepsilon(q)=F(q+\varepsilon{\bf e})$, where ${\bf e}=(1,0,0,\cdots,0)\in\mathbb H^{n}$. If $\varepsilon>0$, then $F_\varepsilon$ is defined in a neighborhood of $\partial \mathcal U_n$. In particular, $F_\varepsilon$ is defined on $\partial \mathcal U_n$. The Hardy space $\mathcal H^2(\mathcal U_n)$ consists of all regular functions $F$ on $\mathcal U_n$, for which
$$\|F\|_{\mathcal H^2(\mathcal U_n)}:=\Big(\sup_{\varepsilon>0}\int_{\partial \mathcal U_n}|F_\varepsilon(q)|^2d\beta(q)\Big)^{1\over 2}<\infty.$$
According to \cite[Theorem 4.1]{CMW}, a function $F\in \mathcal H^2(\mathcal U_n)$ has boundary value $F^b$ that belongs to $L^2(\mathcal U_n)$.

 The identification of  the boundary $\partial \mathcal U_n$ with the quaternionic Heisenberg group $\mathscr H^{n-1}$ helps us to determine the kernel of the Cauchy--Szeg\"o projection from
 $L^2(\partial \mathcal U_n)$ to $\mathcal H^2(\partial \mathcal U_n)$ \cite{CM08}, which was just obtained  recently in
 \cite{CMW}. To be more explicit, we recall the result as follows.

 \medskip\noindent
{\bf Theorem A (\cite{CMW}).} {\it
The Cauchy--Szeg\"o kernel is given by
\begin{align}\label{cauchy-szego}
S(q,p)=s\Big(q_1+\overline p_1-2\sum_{k=2}^n\overline p_kq_k\Big),
\end{align}
for $p=(p_1,p')=(p_1,\cdots, p_n)\in\mathcal U_n$, $q=(q_1,q')=(q_1,\cdots, q_n)\in\mathcal U_n$, where
\begin{align}\label{s}
s(\sigma)=c_{n-1}{\partial^{2(n-1)}\over \partial x_1^{2(n-1)}}{{\overline \sigma}\over |\sigma|^4},\quad
\sigma=x_1+x_2{\bf i}+x_3{\bf j}+x_4{\bf k}\in\mathbb H,
\end{align}
with the real constant $c_{n-1}$ depending only on $n$.
The Cauchy--Szeg\"o kernel satisfies the reproducing property in the following sense
$F(q)=\int_{\partial \mathcal U_n} S(q,\xi)F^b(\xi)d\beta(\xi),\ q\in\mathcal U_n,
$
whenever $F\in \mathcal H^2(\mathcal U_n)$ and $F^b$ its boundary value on $\partial \mathcal U_n$.
}
\medskip

The quaternionic Heisenberg group  $\mathscr H^{n-1}$
 plays the fundamental role in quaternionic analysis and geometry \cite{Chr,Ivanov2,Wa11,shi-wang16}. Its analytic and geometric
behaviours are different from the usual Heisenberg group in many aspects, e.g., there does not exist nontrivial quasiconformal mapping between the quaternionic Heisenberg groups \cite{Pansu} while
quasiconformal mappings between Heisenberg groups are abundant \cite{Koranyi1,Koranyi2}.

The Cauchy--Szeg\"o projection operator $\mathcal C$ can be defined via the ``vertical translate"
from  Cauchy--Szeg\"o kernel for  $ \mathcal U_n$ by
\begin{align*}
(\mathcal C f)(q)=\lim_{\varepsilon\to 0}\int_{\partial \mathcal U_n} S(q+\varepsilon {\bf e}, p)f(p)d\beta(p),\quad
\forall f\in L^2(\partial \mathcal U_n), \quad q\in \partial \mathcal U_n,
\end{align*}
where the limit exists in the $L^2(\partial \mathcal U_n)$ norm and $\mathcal C(f)$ is the boundary limit of some function in $\mathcal H^2(\mathcal U_n)$.

In view of the action of the quaternionic Heisenberg group, the operator $\mathcal C$ can be explicitly described as a convolution operator on this group:
\begin{align}\label{cs projection gp}
(\mathcal C f)(g)=(f*K)(g)=p.v. \int_{\mathscr H^{n-1}}K(h^{-1}\cdot g)f(h)dh,
\end{align}
where the kernel $K(g)$ is defined  in Section 2 below.
We can write
\begin{align}\label{cs projection gp 1}
(\mathcal C f)(g)=p.v. \int_{\mathscr H^{n-1}}K(g,h)f(h)dh,
\end{align}
where $K(g,h)=K(h^{-1}\cdot g)\ {\rm for}\ g\neq h.$ Note that {\eqref{cs projection gp 1} holds whenever $f$ is an $L^2$ function supported in a compact set, for every $g$ outside the support of $f$.}

In \cite{CDLWW}, Chang et al. verify that the kernel $K(g)$ is a standard Calder\'on--Zygmund kernel with respect to the  quasi-metric $\rho$ (defined in Section 2), that is, it satisfies the standard size and smoothness conditions in terms of $\rho$.

 \medskip\noindent
{\bf Theorem B } (\cite{CDLWW}){\bf.} {\it
Suppose $j=1,\ldots, 4n-4,$ and we denote $Y_j$ the standard left-invariant  vector fields on $\mathscr H^{n-1}$ $($defined  as in \eqref{Y} in Section 2$)$. Then we have
\begin{align}\label{gradient}
\left|Y_jK(g)\right|\lesssim{1\over \rho(g,{\bf0})^{Q+1}},\quad g\in\mathscr H^{n-1}\setminus\{{\bf0}\},\quad
\end{align}
where ${\bf0}$ is the neutral element of $\mathscr H^{n-1}$ and $Q=4n+2$ is the homogeneous dimension of $\mathscr H^{n-1}$.

Then we further have
the Cauchy--Szeg\"o kernel $K(g, h)$ on $\mathscr H^{n-1}$ $(g\not=h)$ satisfies the following conditions.
\begin{align*}
&{\rm (i)}\ \ |K(g, h)|\lesssim {1\over\rho(g,h)^{Q}};\\
&{\rm (ii)}\ \ |K(g,h) - K(g_0,h)|\lesssim   {\rho(g, g_0)\over \rho(g_0,h)^{Q+1} },\quad {\rm if}\ \rho(g_0,h)\geq c\rho(g,g_0);\\
&{\rm (iii)}\ \ |K(g,h) - K(g,h_0) |\lesssim   { \rho(h, h_0)\over \rho(g,h_0)^{Q+1} },\quad {\rm if}\ \rho(g,h_0)\geq c \rho(h,h_0)
\end{align*}
for some constant $c>0$,
where  $\rho$ is defined in Section 2.
}

 \medskip\noindent
{\bf Theorem C} (\cite{CDLWW}){\bf.} {\it
The Cauchy--Szeg\"o kernel $K(\cdot, \cdot)$ on $\mathscr H^{n-1}$ satisfies the following pointwise lower bound: there exist a large positive constant $r_0$ and a positive constant $C$ such that
for every $g\in \mathscr H^{n-1}$, there exists a ``twisted truncated sector'' $S_g\subset \mathscr H^{n-1}$ such that
$$ \inf_{g'\in S_g} \rho(g,g')=r_0 $$ and that for every $g_1\in B(g,1)$ and $g_2\in S_g$ we have
\begin{align*}
|K(g_1, g_2)|\geq  {C\over\rho(g_1,g_2)^Q}.
\end{align*}
Moreover, this sector $S_g$ is regular in the sense that $|S_g|=\infty$ and that for every
$R_2>R_1>2r_0$
$$ \big| \big(B(g,R_2)\backslash B(g,R_1)\big)  \cap S_g\big | \approx  \big| B(g,R_2)\backslash B(g,R_1)\big|$$
with the implicit constants  independent of $g$ and $R_1,R_2$.
}

Using the above two theorems, the authors established the characterization of the BMO space and the VMO space via the commutator $[b,\mathcal C]$ in \cite {CDLWW}. It is well-known that the boundedness and compactness of Calder\'on--Zygmund operator commutators on certain function spaces and their characterizations play an important role in various area, such as harmonic analysis, complex analysis, (nonlinear) PDE, etc. 
Recently,  equivalent characterizations
of the boundedness {and compactness} of commutators were further extended to
Morrey spaces over the Euclidean space  by Di Fazio and Ragusa \cite{DiFazioRagusa91BUMIA} and
Chen et al.\,\cite{CDW12CJM},  and by Tao et al.\,\cite{TYY,TYY2} for the Cauchy integral and {Beurling}-Ahlfors transformation commutator, respectively. {Komori and Shirai [20] proved the boundedness of Calder\'on-Zygmund operator commutators with BMO functions over weighted Morrey spaces.}
In this article,
we consider the boundedness and compactness characterizations of
Cauchy--Szeg\"o operator commutator
$[b, \mathcal C]$ on the weighted Morrey spaces over the quaternionic Heisenberg group.
\smallskip

\ Let $p\in (1,\infty)$.  A non-negative function $w\in L^1_{\rm loc}(\cc)$ is in  $A_{p}(\cc)$ if
\begin{align*}
[w]_{A_p(\cc)}:=\sup_{B\subset\mathcal G}\left( \frac{1}{|B|}\int_{B}w(g)dg\right) \left( \frac{1}{|B|}%
\int_{B}w(g) ^{-1/(p-1)}dg\right) ^{p-1}<\infty,
\end{align*}
where the supremum is taken over all balls $B$ in $\cc$.
A non-negative function $w\in L^1_{\rm loc}(\cc)$ is in  $A_{1}(\cc)$ if there exists a constant $C$
such that for all balls $B\subset\cc$,
\begin{equation*}
\frac{1}{|B|}\int_{B}w( g) dg\leq C\mathop{\rm essinf}%
\limits_{x\in B}w(g) .
\end{equation*}%
For $p=\infty$, we define
\begin{equation*}
A_{\infty }(\cc)= \bigcup_{1\leq p<\infty }A_{p}(\cc).
\end{equation*}

Let $p\in(1,\infty)$, $\kappa\in(0,1)$ and $w\in A_p(\cc)$.
The \emph{weighted Morrey space} $L_w^{p,\,\kappa}(\cc)$ (c.f. \cite{KM} )
 is defined by
\begin{equation*}
L_w^{p,\,\kappa}(\cc):=\lf\{f\in L^p_{\rm loc}(\cc):\,\,\|f\|_{L_w^{p,\,\kappa}(\cc)}<\infty\r\}
\end{equation*}
with
\begin{equation*}
\|f\|_{L_w^{p,\,\kappa}(\cc)}:=\sup_{B}
\lf\{\frac{1}{[w(B)]^\kappa}\int_{B}|f(h)|^pw(h)\,dh\r\}^{1/p}.
\end{equation*}

We get the boundedness characterization of Cauchy--Szeg\"o operator commutator.

\begin{thm}\label{CZ com bounded}
Let $p\in(1,\infty)$, $\kappa\in(0,1)$, $w\in A_p(\cc)$ and $b\in L^1_{\loc}(\cc)$.
Then the Cauchy--Szeg\"o operator commutator
$[b, \mathcal C]$ has the following
boundedness characterization:
\begin{enumerate}
    \item[{\rm(i)}]If $b\in {\rm BMO}(\cc)$, then $[b, \mathcal C]$ is bounded on
    $L_w^{p,\,\kappa}(\cc)$.
    \item[{\rm(ii)}]If $b$ is real-valued and $[b, \mathcal C]$ is bounded on
    $L_w^{p,\,\kappa}(\cc)$, then $b\in {\rm BMO}(\cc)$.
\end{enumerate}
\end{thm}

Based on Theorem \ref{CZ com bounded}, we further obtain the compactness characterization
of Cauchy--Szeg\"o operator commutator.
\begin{thm}\label{CZ com compact}
Let $p\in(1,\infty)$, $\kappa\in(0,1)$, $w\in A_p(\cc)$ and $b\in {\rm BMO}(\cc)$.
Then the Cauchy--Szeg\"o operator commutator $[b,\,\mathcal C]$ has the following
compactness characterization:
\begin{enumerate}
    \item[{\rm(i)}]If $b\in {\rm VMO}(\cc)$, then $[b,\,\mathcal C]$ is compact on
    $L_w^{p,\,\kappa}(\cc)$.
    \item[{\rm(ii)}]If $b$ is real-valued and $[b,\,\mathcal C]$ is compact on
    $L_w^{p,\,\kappa}(\cc)$, then $b\in {\rm VMO}(\cc)$.
\end{enumerate}
\end{thm}

This paper is organized as follows. In Section 2  we recall some necessary preliminaries on quaternionic Heisenberg groups. 
In Section 3  we give the proof of Theorem \ref {CZ com bounded}.
The poof of Theorem \ref {CZ com compact} will be provided in Section 4

{\bf Notation:} Throughout this paper,  $C$ will denote positive constant which is independent of the main parameters, but it may vary from line to line. By $f\lesssim g$,
we shall mean $f\le C g$ for some positive constant 
 $C$. If $f\lesssim g$ and $g\lesssim f$, we then write $f \approx g$.

\section{Preliminaries}



Recall that the space $\mathbb H$ of quaternion numbers forms a division algebra with respect to the coordinate addition and the quaternion multiplication
\begin{align*}
x x'&=(x_1+x_2{\bf i}+x_3{\bf j}+x_4{\bf k})(x'_1+x'_2{\bf i}+x'_3{\bf j}+x'_4{\bf k})\\
&=x_1x'_1-x_2x'_2-x_3x'_3-x_4x'_4
+\left(x_1x'_2+x_2x'_1+x_3x'_4-x_4x'_3 \right){\bf i}\\
&\quad +\left(x_1x'_3-x_2x'_4+x_3x'_1+x_4x'_2 \right){\bf j}
+\left(x_1x'_4+x_2x'_3-x_3x'_2+x_4x'_1 \right){\bf k},
\end{align*}
for any $x=x_1+x_2{\bf i}+x_3{\bf j}+x_4{\bf k}$, $x'=x'_1+x'_2{\bf i}+x'_3{\bf j}+x'_4{\bf k}\in \mathbb H$.
The conjugate $\bar x$ is defined by
$$\bar x=x_1-x_2{\bf i}-x_3{\bf j}-x_4{\bf k},$$
 and the modulus $|x|$ is defined by
$$|x|^2=x \bar x=\sum_{j=1}^4x_j^2.$$
 The conjugation inverses the product of quaternion number in the following sense $\overline{q\cdot \sigma}=\bar \sigma \cdot\bar{q}$ for any $q, \sigma\in\mathbb H$.
It is clear that
\begin{equation}\begin{split}\label{im}
\operatorname{Im}(\bar{x}x')&=\operatorname{Im}\{(x_1-x_2{\bf i}-x_3{\bf j}-x_4{\bf k})(x'_1+x'_2{\bf i}+x'_3{\bf j}+x'_4{\bf k})\}\\
&=\left(x_1x'_2-x_2x'_1-x_3x'_4+x_4x'_3 \right){\bf i}
+\left(x_1x'_3+x_2x'_4-x_3x'_1-x_4x'_2 \right){\bf j}\\
&\quad
+\left(x_1x'_4-x_2x'_3+x_3x'_2-x_4x'_1 \right){\bf k}\\
&=:\sum_{\alpha=1}^3\sum_{k,j=1}^4b_{kj}^\alpha x_kx'_j{\bf i}_\alpha,
\end{split}
\end{equation}
where ${\bf i}_1={\bf i}, {\bf i}_2={\bf j}, {\bf i}_3={\bf k}$, and $b_{kj}^\alpha$ is the $(k,j)$ th entry of the following matrices $b^\alpha$:
\begin{equation*}
b^1:=\left( \begin{array}{cccc}
0&1 & 0 & 0\\
-1 & 0 & 0&0\\
0 & 0 & 0&-1\\
0 & 0 & 1&0
\end{array}
\right ),
\quad
b^2:=\left( \begin{array}{cccc}
0&0 & 1 & 0\\
0 & 0 & 0&1\\
-1 & 0 & 0&0\\
0 & -1 & 0 &0
\end{array}
\right ),
\quad
b^3:=\left( \begin{array}{cccc}
0&0 & 0 & 1\\
0 & 0 & -1&0\\
0 & 1 & 0&0\\
-1 & 0 & 0 &0
\end{array}
\right ).
\end{equation*}

The underling vector space of the quaternion space $\mathbb H^n$ is $\mathbb R^{4n}$ and that of the pure imaginary $\operatorname{Im}\mathbb H$ is $\mathbb R^3$.

 The quaternionic Heisenberg group $\mathscr H^{n-1}$  is the space $\mathbb R^{4n-1}=\mathbb R^3\times\mathbb R^{4(n-1)}$, which is the underlying vector space of $\operatorname{Im}\mathbb H\times \mathbb H^{n-1}$, endowed with the non-commutative multiplication
\begin{align}\label{prod}
(t,y)\cdot (t',y')=\left(t+t'+2\operatorname{Im}\langle y, y'\rangle, y+y'\right),
\end{align}
where $t=t_1{\bf i}+t_2{\bf j}+t_3{\bf k}$, $t'=t'_1{\bf i}+t'_2{\bf j}+t'_3{\bf k}\in \operatorname{Im}\mathbb H$, $y, y'\in\mathbb H^{n-1}$, and $\langle \cdot, \cdot\rangle$ is the inner product defined by
$$\langle y, y'\rangle=\sum_{l=1}^{n-1}\overline y_l y'_l, \quad y=(y_1,\cdots, y_{n-1}),\quad y'=(y'_1,\cdots, y'_{n-1})\in\mathbb H^{n-1}.$$
It is easy to check that the identity of $\mathscr H^{n-1}$ is the origin ${\bf 0}:=(0,0)$, and the inverse of $(t, y)$ is given by $(-t, -y)$.

The boundary of quaternionic Siegel upper half-space $\partial \mathcal U_n$ can be identified with
 the quaternionic Heisenberg group $\mathscr H^{n-1}$  via  the projection
\begin{align}\label{p}
\pi: \partial\mathcal U_n &\longrightarrow \operatorname{Im}\mathbb H\times\mathbb H^{n-1},\\
\nonumber(|q'|^2+x_2{\bf i}+x_3{\bf j}+x_4{\bf k}, q')&\longmapsto (x_2{\bf i}+x_3{\bf j}+x_4{\bf k}, q').
\end{align}
Let $d\beta$ be the Lebesgue measure on $\partial\mathcal U_n$ obtained by pulling back the Haar measure on the group $\mathscr H^{n-1}$ by the projection $\pi$.



By \eqref{im}, the multiplication of the quaternionic Heisenberg group in terms of real variables can be written as (cf. \cite{shi-wang})
\begin{align}\label{law}
(t,y)\cdot(t', y')=\bigg(t_\alpha+t'_\alpha+2\sum_{l=0}^{n-1}\sum_{j,k=1}^4b_{kj}^\alpha y_{4l+k}y'_{4l+j}, y+y' \bigg),
\end{align}
where
$t=(t_1, t_2, t_3)$, $t'=(t'_1, t'_2, t'_3)\in\mathbb R^3$, $\alpha=1,2,3$, $y=(y_1,y_2,\cdots, y_{4n-4})$, $y'=(y'_1,y'_2,\cdots, y'_{4n-4})\in\mathbb R^{4n-4}$.

 The following vector fields are left invariant on the quaternionic Heisenberg group by the multiplication laws of the quaternionic Heisenberg group in \eqref{law}:
\begin{align}\label{Y}
 Y_{4l+j} ={\partial\over \partial y_{4l+j}} + 2 \sum_{\alpha=1}^3\sum_{k=1}^4 b_{kj}^\alpha y_{4l+k} {\partial\over\partial t_\alpha},
 \end{align}
and
 $$[Y_{4l'+k}, Y_{4l+j}]=2\delta_{ll'}\sum_{\alpha=1}^3b_{kj}^\alpha{\partial\over \partial t_\alpha},$$
for $l, l'=0,\cdots,n-2$, $j,k=1,\cdots 4$.
 Then the horizontal tangent space at $g\in\mathscr H^{n-1}$, denoted by $H_g$, is spanned by the left invariant vectors $Y_1(g),\cdots, Y_{4n-4}(g)$. For each $g\in\mathscr H^{n-1}$, we fix a quadratic form
 $\langle \cdot ,\cdot\rangle_H$ on $H_g$ with respect to which the vectors $Y_1(g),\cdots, Y_{4n-4}(g)$
 are orthonormal.


For any $p=(t,y)\in \mathscr H^{n-1}$, we can associate the automorphism $\tau_p$ of $\mathcal U_n$:
\begin{align}\label{tau h}
\tau_p: (q_1,q')\longmapsto\left(q_1+|y|^2+t+2\langle y, q' \rangle, q'+y\right).
\end{align}
It is obviously extended to the boundary $\partial\mathcal U_n$. It is easy to see that the action on $\partial \mathcal U_n$ is transitive. In particular, we have
$$\tau_p: (0,0)\longmapsto (|y|^2+t, y).$$
 And we can write each $q\in\partial \mathcal U_n$ as $q=\tau_g(0)$ for a unique $g\in \mathscr H^{n-1}$. In this correspondence we have that
{$d\beta(q)=dg$}, the invariant measure on $\mathscr H^{n-1}$. Similarly, we write $p\in\partial \mathcal U_n$ in the form $p=\tau_h(0)$. Then from \eqref{tau h} we can see that
 \begin{align*}
 S(q+\varepsilon {\bf e}, p)&=S\big(\tau_{h^{-1}}(q+\varepsilon {\bf e}), \tau_{h^{-1}}(p)\big)
 =S\big(\tau_{h^{-1}}(q)+\varepsilon {\bf e}, 0\big)
 =S\big(\tau_{h^{-1}\cdot g}(0)+\varepsilon {\bf e}, 0\big).
 \end{align*}
 Take $K_\varepsilon(g)=S(\tau_g(0)+\varepsilon {\bf e},0)$, and denote by
 {$f(h):=f(\tau_h(0))=f(p)$   and
 $(\mathcal C f)(g):=(\mathcal C f)(\tau_g(0))=(\mathcal C f)(q)$ by abuse of notations}.
 Then
 \begin{align}\label{kepsilon}
 (\mathcal C f)(g)=\lim_{\varepsilon\to 0}\int_{\mathscr H^{n-1}}K_{\varepsilon}(h^{-1}\cdot g)f(h)dh,\quad f\in L^2(\mathscr H^{n-1}),
 \end{align}
 where the limit is taken in $L^2(\mathscr H^{n-1})$.

 Recall that the convolution on $\mathscr H^{n-1}$ is defined as
 $$(f*\tilde f)(g)=\int_{\mathscr H^{n-1}}f(h)\tilde f(h^{-1}\cdot g)dh.$$
Therefore, \eqref {kepsilon} can be formally rewritten as
$$ (\mathcal C f)(g)=(f*K)(g),$$
where $K$ is the distribution given by
$\lim_{\varepsilon\to 0}K_{\varepsilon}.$
Thus, if $g=(t, y)\in \mathscr H^{n-1}$ with $t=t_1{\bf i}+t_2{\bf j}+t_3{\bf k}$, then
\begin{align}\label{kg}
K(g)=\lim_{\varepsilon\to 0}K_{\varepsilon}(g)=s(|y|^2+t).
\end{align}


For any $g=(t, y)\in \mathscr H^{n-1}$, the homogeneous norm of $g$ is defined by
$$\|g\|=\left(|y|^4+\sum_{j=1}^3|t_j|^2\right)^{1\over 4}.$$
Obviously,
$\|g^{-1}\|=\|-g\|=\|g\|$ and $\|\delta_r(g)\|=r\|g\|$, where $\delta_r, r>0,$ is the dilation on $\mathscr H^{n-1}$, which is defined as
$$\delta_r(t,y)=(r^2t, ry).$$

On $\mathscr H^{n-1}$, we define the quasi-distance $$\rho(h,g)=\|g^{-1}\cdot h\|.$$
This is a standard definition in general stratified Lie group, see for example \cite[Chapter 1]{FS}.
 It is clear that
 $\rho$ is symmetric and satisfies the generalized triangle inequality
\begin{align}\label{rho}
\rho(h,g)\leq C_\rho(\rho(h,w)+\rho(w,g)),
\end{align}
for any $h, g, w\in \mathscr H^{n-1}$ and some $C_\rho>0$.
 Using $\rho$, we define the balls $B(g,r)$ in $\mathscr H^{n-1}$ by $B(g,r)=\{h: \rho(h,g)<r \}$. Then
$$|B(g,r)|\approx r^Q,$$
where $Q=4n+2$ is the homogeneous dimension of $\mathscr H^{n-1}$.

We now recall the BMO and VMO spaces.
Note that $\mathscr H^{n-1}$ falls into the scope of homogeneous group, and hence we have the natural  BMO space   in this setting due to Folland and Stein \cite{FS}. To be self-enclosed, we recall the definition of the BMO space.
$${\rm BMO}(\mathscr H^{n-1})=\{ b\in L^1_{loc}(\mathscr H^{n-1}): \|b\|_{ {\rm BMO}(\mathscr H^{n-1}) }<\infty \},$$
where
$$ \|b\|_{ {\rm BMO}(\mathscr H^{n-1}) }= \sup\limits_{B} {1\over |B|} \int_{B} | b(g)-b_B |dg,$$
 where the supremum is taken over all balls $B\subset \mathscr H^{n-1}$ and $b_B={1\over |B|}\int_Bb(g)dg$. 
 {According to \cite[Lemma 5.2]{CRT}, 
 \begin{equation}\label{bmoeq}
  \|b\|_{ {\rm BMO}(\mathscr H^{n-1}) }\approx \|b\|^p_{ {\rm BMO}^p(\mathscr H^{n-1}) },
 \end{equation}
 for $1\leq p<\infty$, where 
 $$\|b\|_{ {\rm BMO}^p(\mathscr H^{n-1}) }= \sup\limits_{B} \bigg({1\over |B|} \int_{B} | b(g)-b_B |^pdg\bigg)^{1\over p}.$$ }
Similarly we also have the VMO space on $\mathscr H^{n-1}$, which is the closure of $C^\infty_c(\mathscr H^{n-1})$ under the norm of $\|\cdot\|_{ {\rm BMO}(\mathscr H^{n-1}) }$, see \cite{CDLW} for more details of the definition and properties of this VMO space in the more general setting, the stratified Lie group.


\section{Boundedness characterization of  Cauchy--Szeg\"o  commutators}\label{s2.5}

In this section, we will give the proof of Theorem \ref{CZ com bounded}. 
%
Here and hereafter, let
$$M (f;B):=\frac 1{|B|}\int_{B}\lf|f(g)-f_{B}\r|\,dg
\quad \mathrm{with}\quad f_{B}=\frac{1}{|B|}\int_B f(g)\,dg.$$
We first recall the \emph{median value} $\alpha_B(f)$ (c.f. \cite{CDLW}).
%
For any real-valued function $f\in L^1_{\rm loc}(\cc)$ and ball $B\subset \cc$,
let $\alpha_B(f)$ be a real number such that
$$\inf_{c\in\rr}\frac{1}{|B|}\int_B|f(g)-c|\,dg$$
is attained. 
Moreover,  it is known that $\alpha_B(f)$ satisfies that
\begin{equation}\label{median value-1}
   \lf|\lf\{g\in B:\ f(g)>\alpha_B(f)\r\}\r|\le \frac{|B|}2
\end{equation}
and
\begin{equation}\label{median value-2}
   \lf|\lf\{g\in B:\ f(g)<\alpha_B(f)\r\}\r|\le \frac{|B|}2.
\end{equation}

Recall
that an absolutely continuous curve $\gamma: [0,1]\to \mathscr H^{n-1}$ is horizontal if its tangent vectors
$\dot{\gamma}(t),  t\in[0, 1]$, lie in the horizontal tangent space $H_{\gamma(t)}$. By \cite{Chow}, any given two points $p,q\in \mathscr H^{n-1}$ can be connected by a horizontal curve.

 The Carnot--Carath\'eodory metric on $\mathscr H^{n-1}$ as follows. For $g,h\in \mathscr H^{n-1}$,
\begin{align}\label{cc metric}
d_{cc}(g,h):= \inf_\gamma \int_0^1 \langle  \dot{\gamma}(t), \dot{\gamma}(t)\rangle_H^{1\over2}\ dt,
\end{align}
where $\gamma: [0,1]\to \mathscr H^{n-1}$ is a horizontal Lipschitz curve with $\gamma(0)=g$, $\gamma(1)=h$.
It is known that the Carnot--Carath\'eodory metric $d_{cc}$ is left-invariant,  and it is equivalent to the homogeneous metric $\rho$ in the sense that: there exist $C_d, \tilde C_d>0$ such that for any $g, h\in\mathscr{H}^{n-1}$ (see \cite[(1.21)]{IV}),
\begin{align}\label{norm-equ}
C_d\rho(g,h)\leq d_{cc}(g,h)\leq \tilde C_d\rho(g,h).
 \end{align}

In order to {prove}  Theorem \ref{CZ com bounded}, we need the following lemmas.
Denote $w (B):=\int_{B}w (g)\,dg$.

\begin{lem}\label{lemlw}
 Let $w\in A_{p}(\cc), p\geq1$. Then there
exist constants $\hat C_{1}, \hat C_{2}, \hat C>0$ and $\sigma\in(0,1)$ such that
\begin{equation*}
\hat C_{1}\left(\frac{|E|}{|B|}\right)^{p}\leq\frac{w(E)}{w\left(B\right)}\leq
\hat C_{2}\left(\frac{|E|}{|B|}\right)^{\sigma}
\end{equation*}
for any measurable subset $E$ of a ball $B$. Especially, for any $\lambda >1$,
\begin{equation*}
w\left( B\left( g_{0},\lambda R\right) \right) \leq \hat C\lambda ^{Qp}w\left(
B\left( g_{0},R\right) \right).
\end{equation*}
\end{lem}

\begin{lem}[\cite {DGKLWY}]\label{wlp}
Let $b\in {\rm BMO}(\cc)$ and $T$ be a Calder\'on-Zygmund operator. If $1<p<\infty$ and $w\in A_p(\cc)$, then
$[b,T]$ is bounded on $L^p_{w}(\cc)$.
\end{lem}

\begin{lem}[\cite{CDLWW}]\label{lb}
Since $K(g,h)$ is $\mathbb H$-valued, we write
$$ K(g,h) = K_1(g,h)+K_2(g,h){\bf i}+K_3(g,h){\bf j}+K_4(g,h){\bf k}, $$
where each $K_i(g,h)$ is real-valued, $i=1,2,3,4$.
Then there is  at least one of the $K_i$ above satisfies the following argument:

There exist positive constants $3\le A_1\le A_2$ such that for any {ball} $B:=B(g_0, r)\subset \mathscr H^{n-1}$, there exist another {ball} $\widetilde B:=B(h_0, r)\subset \mathscr H^{n-1}$
such that $A_1 r\le d_{cc}(g_0, h_0)\le A_2 r$, 
and  for all $(g,h)\in ( B\times \widetilde{B})$,
$K_i(g, h)$ does not change sign
and
\begin{equation}\label{e-assump cz ker low bdd}
|K_i(g, h)|\geq  {C\over\rho(g,h)^{Q}}.
\end{equation}

\end{lem}

\begin{proof}[Proof of Theorem \ref {CZ com bounded}]
(i).  Let $1<p<\infty$. It is sufficient to prove that 
$$\lf\{\frac{1}{[w(B)]^\kappa}\int_{B}|[b,\mathcal C](g)|^pw(g)\,dg\r\}^{1/p}\lesssim  \|b\|_{ {\rm BMO}(\mathscr H^{n-1}) }\|f\|_{L_w^{p,\,\kappa}(\cc)} ,$$
holds for any ball $B$.

Now fix a ball $B=B(g_0, r)$ and decompose $f=f\chi_{2B}+f\chi_{\cc\setminus 2B}=:f_1+f_2$. Then
\begin{align*}
&{1\over w(B)^{\kappa}}\int_B\left|[b,\mathcal C] f(g)  \right|^pw(g)dg\\
&\lesssim \bigg\{ {1\over w(B)^{\kappa}}\int_B\left|[b,\mathcal C] f_1(g)  \right|^pw(g)dg+{1\over w(B)^{\kappa}}\int_B\left|[b,\mathcal C] f_2(g)  \right|^pw(g)dg\bigg\}\\
&=:I+II.
\end{align*}

For the term $I$, by Lemma \ref{wlp}, we can obtain
\begin{align*}
{1\over w(B)^{\kappa}}\int_B\left|[b,\mathcal C] f_1(g)  \right|^pw(g)dg
&\leq{1\over w(B)^{\kappa}} \int_{\cc}\left|[b,\mathcal C] f_1(g)\right|^pw(g)dg\\
&\lesssim \|b\|^p_{ {\rm BMO}(\mathscr H^{n-1}) }{1\over w(B)^{\kappa}}\int_{2B}|f(g)|^pw(g)dg\\
&\lesssim\|b\|^p_{ {\rm BMO}(\mathscr H^{n-1}) }\|f\|^p_{L_w^{p,\,\kappa}(\cc)}.
\end{align*}
Thus, we have
$$\|[b,\mathcal C] f_1\|_{L_w^{p,\,\kappa}(\cc)}\lesssim\|b\|^p_{ {\rm BMO}(\mathscr H^{n-1}) }\|f\|^p_{L_w^{p,\,\kappa}(\cc)}.$$

For the term $II$, observe that for $g\in B$, by Theorem B, we have
\begin{align*}
&\left|[b,\mathcal C] f_2(g)\right|^p\\
&\leq\bigg(\int_{\cc}|b(g)-b(u)| |K(g,u)| |f_2(u)|du\bigg)^p\\
&\lesssim \bigg( \int_{\cc\setminus 2B}{|b(g)-b(u)|\over\rho(g,u)^Q} |f(u)|du\bigg)^p\\
&\lesssim \bigg(\int_{\cc\setminus 2B}{ |f(u)|\over\rho(g_0,u)^Q}\left\{\left|b(g)-b_{B,w}\right|+\left|b_{B,w}-b(u)\right|\right\}du\bigg)^p\\
&\lesssim\bigg(\int_{\cc\setminus 2B}{ |f(u)|\over\rho(g_0,u)^Q}du \bigg)^p \left|b(g)-b_{B,w}\right|^p
+\bigg(\int_{\cc\setminus 2B}{ |f(u)|\over\rho(g_0,u)^Q}\left|b_{B,w}-b(u)\right|du \bigg)^p,
\end{align*}
where 
$b_{B,w}={1\over w(B)}\int_{B}b(g)w(g)dg.$
Then we have
\begin{align*}
&{1\over w(B)^\kappa}\int_B\left|[b,\mathcal C] f_2(g)\right|^pw(g)dg\\
&\lesssim 
{1\over w(B)^\kappa} \bigg(\int_{\cc\setminus 2B}{ |f(u)|\over\rho(g_0,u)^Q}du \bigg)^p\int_B\left|b(g)-b_{B,w}\right|^pw(g)dg  \\
&\quad+\bigg(\int_{\cc\setminus 2B}{ |f(u)|\over\rho(g_0,u)^Q}\left|b_{B,w}-b(u)\right|du \bigg)^pw(B)^{1-\kappa}\\
&=:III+IV.
\end{align*}

For $III$, by the H\"older inequality, Theorem 3.5 in \cite{Ho} and Lemma \ref{lemlw}, we have
\begin{align*}
III&\lesssim \|f\|^p_{L_w^{p,\kappa}(\cc)}{1\over w(B)^\kappa}\bigg(\sum_{j=1}^{\infty} {1\over w(2^{j+1}B)^{{1-\kappa\over p}}} \bigg)^p\int_B\left|b(g)-b_{B,w}\right|^pw(g)dg\\
&\lesssim\|b\|_{{\rm BMO}(\cc)}\|f\|^p_{L_w^{p,\kappa}(\cc)}\bigg(\sum_{j=1}^{\infty} {w(B)^{1-\kappa\over p}\over w(2^{j+1}B)^{{1-\kappa\over p}}} \bigg)^p\\
&\lesssim\|b\|_{{\rm BMO}(\cc)}\|f\|^p_{L_w^{p,\kappa}(\cc)}.
\end{align*}

For $IV$, by the H\"older inequality, we can get
\begin{align*}
IV&\lesssim \bigg(\sum_{j=1}^{\infty}{1\over |2^jB|}\int_{2^{j+1}B}|f(u)| \left| b(u)-b_{B,w}\right| du \bigg)^p w(B)^{1-\kappa}\\
&\lesssim\bigg\{\sum_{j=1}^{\infty}{1\over |2^jB|}\left(\int_{2^{j+1}B}|f(u)|^p w(u)du\right)^{1\over p}\\
&\qquad \times\left(\int_{2^{j+1}B} \left| b(u)-b_{B,w}\right| ^{p'}w(u)^{1-p'}du\right)^{1\over p'} \bigg\}^p
 w(B)^{1-\kappa}\\
 &\lesssim \|f\|^p_{L_w^{p,\kappa}(\cc)}\bigg\{\sum_{j=1}^{\infty}{w(2^{j+1}B)^{\kappa\over p}\over |2^jB|}
\left(\int_{2^{j+1}B} \left| b(u)-b_{B,w}\right| ^{p'}w(u)^{1-p'}du\right)^{1\over p'} \bigg\}^p
 w(B)^{1-\kappa}.
\end{align*}

Observe that
\begin{align*}
&\left(\int_{2^{j+1}B} \left| b(u)-b_{B,w}\right| ^{p'}w(u)^{1-p'}du\right)^{1\over p'} \\
&\leq\left(\int_{2^{j+1}B} \left\{\left| b(u)-b_{2^{j+1}B, w^{1-p'}}\right|+\left|b_{2^{j+1}B, w^{1-p'}}-b_{B,w}\right| \right\} ^{p'}w(u)^{1-p'}du\right)^{1\over p'} \\
&\leq\left(\int_{2^{j+1}B} \left| b(u)-b_{2^{j+1}B, w^{1-p'}}\right| ^{p'}w(u)^{1-p'}du\right)^{1\over p'}
+\left|b_{2^{j+1}B, w^{1-p'}}-b_{B,w}\right| \left(\int_{2^{j+1}B}w(u)^{1-p'}du\right)^{1\over p'}\\
&=:V+VI.
\end{align*}

Since $w\in A_p(\cc)$, we have $w^{1-p'}\in A_{p'}(\cc)$, where $1/p+1/p'=1$. By Theorem 5.5 in \cite{Ho}, we can obtain that
$$V\lesssim\|b\|_{{\rm {BMO}}(\cc)}w^{1-p'}(2^{j+1}B)^{1\over p'}.$$

For $VI$,  by Theorem 5.5 in \cite{Ho}, we have
\begin{align*}
\left|b_{2^{j+1}B, w^{1-p'}}-b_{B,w}\right|&\leq \left|b_{2^{j+1}B, w^{1-p'}}-b_{2^{j+1}B}\right|+\left|b_{2^{j+1}B}-b_{B}\right|
+\left|b_{B}-b_{B,w}\right|\\
&\leq{1\over w^{1-p'}(2^{j+1}B)}\int_{2^{j+1}B}\left|b(u)-b_{2^{j+1}B}\right|w(u)^{1-p'}du\\
& +2^Q(j+1)\|b\|_{{\rm {BMO}}(\cc)}
+{1\over w(B)}\int_{B}\left|b(u)-b_{B}\right|w(u)du.
\end{align*}

Since $b\in {\rm {BMO}}(\cc)$, by the John-Nirenberg inequality (c.f. {\cite[Proposition 6]{Kro}}), there exist $\bar C_1>0$ and $\bar C_2>0$ such that for any ball $B$ and 
$\alpha>0$,
$$\left| \left\{g\in B: |b(g)-b_B|>\alpha \right\}\right|\leq \bar C_1|B|e^{-{\bar C_2\alpha\over\|b\|_{{\rm {BMO}}(\cc)}}}.$$
Then by Lemma \ref {lemlw}, we have
$$w\left( \left\{g\in B: |b(g)-b_B|>\alpha \right\}\right)\leq \bar C_1w(B)e^{-{\bar C_2\alpha\sigma\over\|b\|_{{\rm {BMO}}(\cc)}}}, $$
for some $\sigma\in (0,1)$. Therefore,
\begin{align*}
\int_B |b(u)-b_B|w(u)du&=\int_{0}^{\infty}w\left(\{g\in B: |b(g)-b_B|>\alpha \} \right)d\alpha\\
&\lesssim w(B)\int_{0}^{\infty}e^{-{\bar C_2\alpha\sigma\over\|b\|_{{\rm {BMO}}(\cc)}}} d\alpha\\
&\lesssim w(B)\|b\|_{{\rm {BMO}}(\cc)}.
\end{align*}
Similarly, we have
$$\bigg(\int_{2^{j+1}B}\left|b(u)-b_{2^{j+1}B}\right|w(u)^{1-p'}du\bigg)^{1\over p'}\lesssim (j+1)\|b\|_{{\rm {BMO}}(\cc)}w^{1-p'}(2^{j+1}B)^{1/p'}.$$
Together with Lemma \ref {lemlw}, we have 
\begin{align*}
IV&\lesssim \|f\|^p_{L_w^{p,\kappa}(\cc)}\|b\|^p_{{\rm {BMO}}(\cc)}\bigg[\sum_{j=1}^{\infty}{w(2^{j+1}B)^{\kappa\over p}\over |2^jB|}(j+1)w^{1-p'}(2^{j+1}B)^{1/p'}
\bigg]^p
 w(B)^{1-\kappa}\\
 &\lesssim \|f\|^p_{L_w^{p,\kappa}(\cc)}\|b\|^p_{{\rm {BMO}}(\cc)}\bigg[\sum_{j=1}^{\infty}{(j+1)w(B)^{1-k\over p}\over w(2^{j+1}B)^{1-\kappa\over p}}
 \bigg]^p\\
 &\lesssim \|f\|^p_{L_w^{p,\kappa}(\cc)}\|b\|^p_{{\rm {BMO}}(\cc)}\bigg[\sum_{j=1}^{\infty}(j+1) 2^{-(j+1)(1-\kappa)Q\sigma\over p}
 \bigg]^p\\
&\lesssim \|f\|^p_{L_w^{p,\kappa}(\cc)}\|b\|^p_{{\rm {BMO}}(\cc)} .
\end{align*}

As a consequence, we have
\begin{align*}
\|[b,\mathcal C]f_2\|_{L_w^{p,\kappa}(\cc)}\lesssim\|f\|_{L_w^{p,\kappa}(\cc)}\|b\|_{{\rm {BMO}}(\cc)} .
\end{align*}
This completes the proof.

  (ii). To prove that $b\in{\rm{BMO}}(\cc)$, it suffices to show that,
  for any ball $B\subset\cc$, $M(b;B)\lesssim 1.$
  Let $B=B(g_0,r)$ be a ball in $\cc$. Let  $\widetilde B:=B(h_0, r)\subset \mathscr H^{n-1}$ be the ball in Lemma \ref{lb}.
 Let
  $$E_1:=\lf\{g\in B:\ b(g)\geq\alpha_{\widetilde{B}}(b)\r\}\quad\mathrm{and}
  \quad E_2:=\lf\{g\in B:\ b(g)<\alpha_{\widetilde{B}}(b)\r\};$$
 {$$F_1=\lf\{u\in\widetilde{B}:\ b(u)\leq\alpha_{\widetilde{B}}(b)\r\}\quad\mathrm{and}
  \quad F_2=\lf\{u\in\widetilde{B}:\ b(u)\geq\alpha_{\widetilde{B}}(b)\r\}.$$
  From \eqref {median value-1} and \eqref {median value-2} we can see
  $|F_1|\geq{1\over 2}|\tilde B|={1\over 2}|B|$ and $|F_2|\geq{1\over 2}|\tilde B|={1\over 2}|B|$.}
  For any $(g,u)\in E_j\times F_j,\,j\in\{1,2\}$,
  $$|b(g)-b(u)|=\lf|b(g)-\alpha_{\widetilde{B}}(b)\r|+\lf|\alpha_{\widetilde{B}}(b)-b(u)\r|
  \geq\lf|b(g)-\alpha_{\widetilde{B}}(b)\r|.$$

   Since $b$ is real-valued,
  from {Lemma \ref{lb}}, the H\"{o}lder inequality and the boundedness of $[b,\mathcal{C}]$
  on $L^{p,\,\kappa}_w(\cc)$, we deduce that
  \begin{align*}
   M(b;B)&\lesssim\frac{1}{|B|}\int_B\lf|b(g)-\alpha_{\widetilde{B}}(b)\r|\,dg
    \approx\sum_{j=1}^2\frac{1}{|B|}\int_{E_j}\lf|b(g)-\alpha_{\widetilde{B}}(b)\r|\,dg\\
    &\lesssim\sum_{j=1}^2\frac{1}{|B|}\int_{E_j}\int_{F_j}
    \frac{|b(g)-\alpha_{\widetilde{B}}(b)|}{|B|}\,du\,dg
    \approx\sum_{j=1}^2\frac{1}{|B|}\int_{E_j}\int_{F_j}
    \frac{|b(g)-\alpha_{\widetilde{B}}(b)|}{\rho(g,u)^{Q}}\,du\,dg\\
    &\lesssim\sum_{j=1}^2\frac{1}{|B|}\int_{E_j}\int_{F_j}
    |b(g)-b(u)|\frac{1}{{\rho(g,u)^{Q}}}\,du\,dg\\
    &\lesssim\sum_{j=1}^2\frac{1}{|B|}\int_{E_j}\lf|\int_{F_j}
    [b(g)-b(u)]K_i(g,u)\,du\r|\,dg
    \sim\sum_{j=1}^2\frac{1}{|B|}\int_{E_j}\lf|[b,\mathcal{C}]\chi_{F_j}(g)\r|\,dg\\
    &\lesssim\sum_{j=1}^2\frac{1}{|B|}\int_{B}\lf|[b,\mathcal{C}]\chi_{F_j}(g)\r|\,dg
    \lesssim\sum_{j=1}^2\frac{1}{|B|}\lf\|[b,\mathcal C]\chi_{F_j}\r\|_{L_w^{p,\,\kappa}(\cc)}
    [w(B)]^{\frac{\kappa-1}{p}}|B|\\
    &\lesssim\sum_{j=1}^2\lf\|[b,\mathcal C]\r\|_{L_w^{p,\,\kappa}(\cc)\to
    L_w^{p,\,\kappa}(\cc)}\lf\|\chi_{F_j}\r\|_{L_w^{p,\,\kappa}(\cc)}
    [w(B)]^{\frac{\kappa-1}{p}}\\
    &\lesssim\|[b,\mathcal C]\|_{L_w^{p,\,\kappa}(\cc)\to
    L_w^{p,\,\kappa}(\cc)}\lf[w\lf(\widetilde{B}\r)\r]^{\frac{1-\kappa}{p}}
    [w(B)]^{\frac{\kappa-1}{p}}\\
    &\lesssim\|[b,\mathcal C]\|_{L_w^{p,\,\kappa}(\cc)\to L_w^{p,\,\kappa}(\cc)}.
  \end{align*}
  This finishes the proof of Theorem \ref{CZ com bounded}.
\end{proof}

\section{Compactness characterization of Cauchy--Szeg\"o  commutators}\label{s3}

In this section, we will give the proof of Theorem \ref{CZ com compact}.

\subsection{Proof of Theorem \ref{CZ com compact}(i)}\label{s3.1}
We first  give a sufficient condition for  subsets of weighted Morrey spaces
to be relatively compact. Recall that
a subset $\cf$ of $L_w^{p,\,\kappa}(\cc)$ is said to be \emph{totally bounded}
(or \emph{relatively compact}) if the $L_w^{p,\,\kappa}(\cc)$-closure of
$\cf$ is compact.

\begin{lem}\label{l-fre kol}
For any $p\in(1,\infty)$, $\kappa\in(0,1)$ and $w\in A_p(\cc)$, a subset $\cf$ of
$L_w^{p,\,\kappa}(\cc)$ is totally bounded (or relatively compact) if
the set $\cf$ satisfies the following three conditions:
\begin{itemize}
\item[{\rm(i)}] $\cf$ is bounded, namely, $$\sup_{f\in\cf}\|f\|_{L_w^{p,\,\kappa}(\cc)}<\infty;$$
\item[{\rm(ii)}] $\cf$ uniformly vanishes at infinity, namely, for any $\epsilon\in(0,\infty)$,
there exists some positive constant $M$ such that, for any $f\in\cf$,
$$\lf\|f\chi_{\{g\in\cc:\ \|g\|>M\}}\r\|_{L_w^{p,\,\kappa}(\cc)}<\epsilon;$$

\item[{\rm(iii)}] $\cf$ is uniformly equicontinuous, namely, for any $\epsilon\in(0,\infty)$,
there exists some positive constant $\beta$ such that,
for any $f\in\cf$ and $\xi\in \cc$ with $\|\xi\|\in[0,\beta)$,
$$\|f(\cdot \xi)-f(\cdot)\|_{L_w^{p,\,\kappa}(\cc)}<\epsilon.$$
\end{itemize}
\end{lem}

The proof of this lemma follows from \cite[Theorem 1.5]{MaoSunWu16ActaMathSin} by a minor modification from Euclidean setting 
to quaternionic Heisenberg group, since it only requires the following key elements of the underlying space: metric and doubling measure.

Before we give the proof of Theorem \ref{CZ com compact},
we first need to establish the boundedness of the maximal operator $\mathcal C_\ast$
of a family of smooth truncated Cauchy--Szeg\"o transforms
$\{\mathcal C_\eta\}_{\eta\in(0,\infty)}$ as follows.
For $\eta\in(0,\infty)$, let
$$\mathcal C_\eta f(g):=\int_{\cc} K_{ \eta}(g, u)f(u)\,du,$$
where the kernel
$K_{ \eta}(g,u):=K(g,u)\varphi(\frac{\rho(g,u)}\eta)$ with
 $\varphi\in C_c^{\infty}(\rr)$ satisfying that
 \begin{align*}
 \phi_\varepsilon(g)=
 \begin{cases}
\varphi(t)\equiv 0\quad&{\rm if}\ \  t\in(-\fz, \frac12),\\
 \varphi(t)\in[0,1],  \quad &{\rm if}\ \ t\in[\frac12, 1],\\
 \varphi(t)\equiv 1, \quad &{\rm if}\ \ t\in (1, \infty).
 \end{cases}
\end{align*}
Let $$[b, \mathcal C_\eta]f(g):=\int_{\cc} [b(g)-b(u)]K_{ \eta}(g, u)f(u)\,du.$$

The\emph{ maximal operator $\mathcal C_{\ast}$} is
defined by setting, for any suitable function $f$ and $g\in\cc$,
$$\mathcal C_{\ast} f(g):=\sup_{\eta\in(0,\infty)}
\lf|\int_{\cc}K_{ \eta}(g, u)f(u)\,du\r|.$$
Recall that
the \emph{Hardy-Littlewood maximal operator $\cm$} is defined by
$$\cm f(g):=\sup_{B\ni g}\frac{1}{|B|}\int_B|f(u)|\,du,$$
for any $f\in L^1_{\loc}(\cc)$ and $g\in\cc$, where the supremum is taken over all the balls $B$ of $\cc$ that contain $g$.

We now recall the mean value theorem on homogeneous groups (c.f. \cite[Theorem 1.41]{FS}), which covers the quaternionic Heisenberg group.

\begin{lem} \label{mean-value}
There exist $C>0$ and $\gamma>0$ such that for any $f\in C^1(\cc)$ and $g, u\in\cc$,
$$\left|f(g\cdot u)-f(g) \right|\leq C\|u\|\sup_{\|\xi\|\leq \gamma\|u\|,1\leq j\leq 4n-4}|Y_jf(g\cdot\xi)|.$$
\end{lem}

Denote $\nabla_H=(Y_1,\cdots, Y_{4n-4})$. 
We have the following conclusions.
\begin{lem}\label{l-approx commuta}
There exists a positive constant $C$ such that,
for any $b\in C^\fz_c(\cc)$, $\eta\in(0,\infty)$, $f\in L^1_{\loc}(\cc)$ and $g\in\cc$,
$$\lf|\lf[b, \mathcal C_\eta\r]f(g)-\lf[b, \mathcal C\r]f(g)\r|
\le C\eta
\lf\|\nabla_H b\r\|_{L^\infty(\cc)} \cm f(g).$$
\end{lem}

\begin{proof}
Let $f\in L^1_{\loc}(\cc)$. For any $g\in\cc$, we have
\begin{align*}
&\lf|\lf[b, \mathcal C_\eta\r]f(g)-\lf[b, \mathcal C\r]f(g)\r|\\
&\quad=\lf|\int_{\eta/2<\rho(g,u)\leq\eta} [b(g)-b(u)]K(g,u)f(u)\,du
    -\int_{\rho(g,u)\leq\eta} [b(g)-b(u)]K(g,u)f(u)\,du\r|\\
&\quad\ls\int_{\rho(g,u)\leq\eta} |b(g)-b(u)|\lf|K(g,u)\r||f(u)|\,du.
\end{align*}
From the smoothness of $b$, Lemma \ref{mean-value} and Theorem B, we deduce that
\begin{align*}
&\int_{\rho(g,u)\leq\eta} |b(g)-b(u)|\lf|K(g, u)\r||f(u)|\,du\\
&\ls \lf\|\nabla_H b\r\|_{L^\infty(\cc)}\sum_{j=0}^\infty
\int_{\frac\eta{2^{j+1}}<\rho(g,u)\leq\frac\eta{2^j}} \frac{\rho(g,u)}{\rho(g,u)^{Q}}|f(u)|\,du\\
&\ls \eta\lf\|\nabla_H b\r\|_{L^\infty(\cc)}\cm f(g),
\end{align*}
which completes the proof of Lemma \ref{l-approx commuta}.
\end{proof}

\begin{lem}\label{l-sub opr bdd}
Let $p\in(1,\infty)$, $\kappa\in(0,1)$ and $w\in A_p(\cc)$.
Then there exists a positive constant $C$ such that,
for any $f\in L_w^{p,\,\kappa}(\cc)$,
$$\lf\|\mathcal C_{\ast} f\r\|_{L_w^{p,\,\kappa}(\cc)}
+\|\cm f\|_{L_w^{p,\,\kappa}(\cc)}\leq C\|f\|_{L_w^{p,\,\kappa}(\cc)}.$$
\end{lem}

\begin{proof}
For the boundedness of $\cm$ on $L_w^{p,\,\kappa}(\cc)$ one can refer to
\cite{AraiMizuhara97MathNachr}. We only consider the boundedness of $\mathcal C_\ast$.
For any fixed ball $B\subset \cc$ and $f\in L_w^{p,\,\kappa} (\cc)$, we write
$$f:=f_1+f_2:=f\chi_{2B}+ f\chi_{\cc\setminus 2B}.$$
Observe   $f_1\in L_w^{p} (\cc)$.
Then, from the boundedness of $\mathcal C_\ast$ on $L_w^{p} (\cc)$ (see, for example,
\cite[Theorem 1.1]{HSZ}),  the H\"older inequality, Theorem B, we deduce that
\begin{align*}
&\lf[\int_B|\mathcal C_\ast f(g)|^pw(g)\,dg\r]^{\frac1p}\\
&\quad\ls \lf[\int_B|\mathcal C_\ast f_1(g)|^pw(g)\,dg\r]^{\frac1p}
  +\sum_{k=1}^\fz\lf\{\int_B\lf[\int_{2^{k+1}B\setminus 2^kB}
  \frac{|f(u)|}{ \rho(g,u)^Q}\,du\r]^pw(g)\,dg\r\}^{\frac1p}\\
&\quad\ls \lf[\int_{2B}|f(g)|^pw(g)\,dg\r]^{\frac1p}+
  \sum_{k=1}^\fz \lf[\frac{w(B)}{|2^kB|^p}\lf\{\int_{2^{k+1}B}|f(u)|[w(u)]^{\frac1p}
  [w(u)]^{-\frac1p}\,du\r\}^p\r]^{\frac1p}\\
&\quad\ls \|f\|_{L_w^{p,\,\kappa} (\cc)}[w(B)]^{\frac\kappa p}+\sum_{k=1}^\fz \lf\{w(B)\lf[w\lf(2^kB\r)\r]^{\kappa-1}\|f\|^p_{L_w^{p,\,\kappa} (\cc)}\r\}^{\frac1p}\\
&\quad\ls \|f\|_{L_w^{p,\,\kappa} (\cc)}w(B)^{\frac\kappa p}+\sum_{k=1}^\fz \lf\{[w(B)]^\kappa2^{2k\sigma(\kappa-1)}\|f\|^p_{L_w^{p,\,\kappa} (\cc)}\r\}^{\frac1p}\\
&\quad\ls \|f\|_{L_w^{p,\,\kappa} (\cc)}[w(B)]^{\frac\kappa p},
\end{align*}
where, in the fourth inequality, we used Lemma \ref{lemlw} with some $\sigma\in(0, 1)$.
This finishes the proof of Lemma \ref{l-sub opr bdd}.
\end{proof}

\begin{proof}[Proof of Theorem \ref{CZ com compact}(i)]
When $b\in \cmoc$, for any $\varepsilon\in(0,\infty)$,
there exists $b^{(\varepsilon)}\in C^\fz_c(\cc)$ such that
$\|b-b^{(\varepsilon)}\|_{\bmoc}<\varepsilon.$
Then, from the boundedness of the commutator $[b, \mathcal C]$ on $L_w^{p,\,\kappa} (\cc)$, we obtain
\begin{align*}
\lf\|\lf[b,\mathcal C\r]f-\lf[b^{(\varepsilon)},\mathcal C\r]f\r\|_{L_w^{p,\,\kappa} (\cc)}
&=\lf\|\lf[b-b^{(\varepsilon)},\mathcal C\r]f\r\|_{L_w^{p,\,\kappa} (\cc)}\\
&\lesssim \lf\|b-b^{(\varepsilon)}\r\|_{\bmoc}\|f\|_{L_w^{p,\,\kappa} (\cc)}\\
&<\varepsilon\|f\|_{L_w^{p,\,\kappa} (\cc)}.
\end{align*}
Moreover, by using Lemmas \ref{l-approx commuta} and \ref{l-sub opr bdd},
we get
$$\lim_{\eta\to0}\lf\|\lf[b, \mathcal C_\eta\r]-
  \lf[b, \mathcal C\r]\r\|_{L_w^{p,\,\kappa} (\cc)\to L_w^{p,\,\kappa} (\cc)}=0.$$
Now it suffices to show that, for any $b\in C^\fz_c(\cc)$ and $\eta\in(0,\infty)$
small enough, $[b,\,\mathcal C_\eta]$ is a compact operator on $L_w^{p,\,\kappa} (\cc)$,
which is equivalent to show that, for any bounded subset
$\cf\subset L_w^{p,\,\kappa} (\cc)$, $[b,\,\mathcal C_\eta]\cf$ is relatively compact. That is, we need to verify
$[b,\,\mathcal C_\eta]\cf$ satisfies the conditions (i) through (iii) of Lemma \ref{l-fre kol}.

We first point out that, by \cite[Theorem 3.4]{KomoriShirai09MathNachr}
and the fact that $b\in\bmoc$, we know that
$[b,\,\mathcal C_\eta]$ is bounded on $L_w^{p,\,\kappa} (\cc)$ for the given $p\in(1,\infty)$,
$\kappa\in(0,1)$ and $w\in A_p(\cc)$, which implies that
$[b,\,\mathcal C_\eta]\cf$ satisfies condition (i) of Lemma \ref{l-fre kol}.

Next, since $b\in C^\fz_c(\cc)$, we may further assume
$\|b\|_{L^{\infty}(\cc)}+\|\nabla_H b\|_{L^{\infty}(\cc)}=1$.
Observe that there exists a positive constant $R_0$ such that $\supp (b)\subset B(0,R_0)$.
Let $M\in(10R_0,\infty)$. Thus, for any $u\in B(0,R_0)$ and $g\in\cc$ with $\|g\|\in(M,\infty)$,
we have $\rho(g,u)\sim \|g\|$.
Then, for $g\in\cc$ with $\|g\|>M$, by Theorem B and the H\"{o}lder inequality,
we conclude that
\begin{align*}
\lf|\lf[b,\,\mathcal C_\eta\r]f(g)\r|&\le\int_{\cc}|b(g)-b(u)|\lf|K_{ \eta}(g, u)\r|
  |f(u)|\,du\\
&\ls\|b\|_{L^{\infty}(\cc)}\int_{B(0,\,R_0)}\frac{|f(u)|}{\rho(g,u)^Q}\,du\\
&\ls\frac{1}{\|g\|^Q}\|b\|_{L^{\infty}(\cc)}\lf[\int_{B(0,\,R_0)}|f(u)|^pw(u)\,du\r]^
 {\frac1p}\lf\{\int_{B(0,\,R_0)}[w(u)]^{-\frac{p'}p}\,du\r\}^{\frac1{p'}}\\
&\ls\frac{1}{\|g\|^Q}\|f\|_{L_w^{p,\,\kappa} (\cc)}
 \lf[w(B(0,\,R_0))\r]^{\frac{\kappa-1}p}|B(0,\,R_0)|.
\end{align*}
Therefore, for any fixed ball $B:=B(\wz g,\,\wz r)\subset \cc$, by Lemma \ref{lemlw}, we have
\begin{align*}
&\frac{1}{[w(B)]^\kappa}\int_{B\cap\{g\in\cc:\ \|g\|>M\}}
    \lf|\lf[b,\,\mathcal C_\eta\r]f(g)\r|^p w(g)\,dg\\
&\quad\ls\frac{\|f\|^p_{L_w^{p,\,\kappa} (\cc)}[w(B(0,R_0))]^{\kappa-1}
    |B(0,R_0)|^p}{[w(B)]^\kappa}\\
    &\quad\quad \times\sum_{j=0}^\infty
    \frac{w(B\cap\{g\in\cc:\ 2^jM<\|g\|\leq 2^{j+1}M\})}{|2^jM|^{Qp}}\\
&\quad\ls \|f\|^p_{L_w^{p,\,\kappa} (\cc)}[w(B(0,R_0))]^{\kappa-1}|B(0,R_0)|^p
   \sum_{j=0}^\infty\frac{[w(B(0, 2^jM))]^{1-\kappa}}{|2^jM|^{Qp}}\\
&\quad\ls\|f\|^p_{L_w^{p,\,\kappa} (\cc)}[w(B(0,R_0))]^{\kappa-1}|B(0,R_0)|^p
   \frac{[w(B(0, M))]^{1-\kappa}}{M^{Qp}}\sum_{j=0}^\infty\frac{2^{Qjp(1-\kappa)}}{2^{Qjp}}\\
&\quad   \ls\lf(\frac{R_0}{M}\r)^{kQp}\|f\|^p_{L_w^{p,\,\kappa} (\cc)}.
\end{align*}
Thus, we conclude that
$$\lf\|\lf([b,\mathcal C_\eta]f\r)\chi_{\{g\in\cc:\ \|g\|>M\}}\r\|_{L_w^{p,\,\kappa}(\cc)}
  \ls \lf(\frac{R_0}{M}\r)^{kQ}\|f\|_{L_w^{p,\,\kappa} (\cc)}.$$
Therefore, condition (ii) of Lemma \ref{l-fre kol} holds for
$[b,\mathcal C_\eta]\mathcal{F}$ with $M$ large enough.

It remains to prove that $[b,\mathcal C_\eta]\mathcal{F}$ also satisfies condition (iii) of
Lemma \ref{l-fre kol}.
Let $\eta$ be a fixed positive constant small enough and $\xi\in\cc$ with $\|\xi\|\in(0,\eta/8(1+C_\rho))$.
Then, for any $g\in\cc$, we have
\begin{align*}
&\lf[b, \mathcal C_\eta\r]f(g)-\lf[b, \mathcal C_\eta\r]f(g\cdot\xi)\\
&\quad=[b(g)-b(g\cdot\xi)]\int_{\cc}K_{ \eta}(g, u)f(u)\,du\\
&\quad\quad +\int_{\cc}\lf[K_{ \eta}(g, u)-
  K_{ \eta}(g\cdot\xi, u)\r][b(g\cdot\xi)-b(u)]f(u)\,du\\
&\quad=:L_1(g)+L_2(g).
\end{align*}

Since $b\in C^\fz_c(\cc)$, by Lemma \ref{mean-value}, it follows that, for any $g\in\cc$,
\begin{align*}
|L_1(g)|=|b(g)-b(g\cdot\xi)|\lf|\int_{\cc}K_{ \eta}(g, u)f(u)\,du\r|
 \ls \|\xi\| \lf\|\nabla_H b\r\|_{L^{\infty}(\cc)}\mathcal C_\ast(f)(g).
\end{align*}
Then Lemma \ref{l-sub opr bdd} implies
$\|L_1\|_{L_w^{p,\,\kappa} (\cc)}\ls\|\xi\| \|f\|_{L_w^{p,\,\kappa} (\cc)}$.

To estimate $ L_2(g)$, we first observe that
$K_{ \eta}(g,u)=0$, $K_{ \eta}(g\cdot\xi, u)=0$
for any $g$, $u$, $\xi\in\cc$ with $\rho(g,u)\in(0,\eta/ 4(1+C_\rho))$ and $\|\xi\|\in(0, \eta/8(1+C_\rho))$.
Moreover, by the definition of $K_{ \eta}(g,u)$ and Theorem B,
we know that, for any $g$, $u\in\cc$ with $\rho(g,u)\in[\eta/4(1+C_\rho),\infty)$,
$$\lf|K_{ \eta}(g,u)-K_{ \eta}(g\cdot\xi,u)\r|\ls \frac{\|\xi\|}{\rho(g,u)^{Q+1}}.$$
This in turn implies that, for any $g\in\cc$,
\begin{align*}
|L_2(g)|&\ls \|\xi\|\int_{\rho(g,u)>\eta/4(1+C_\rho)}\frac{|f(u)|}{\rho(g,u)^{Q+1}}\,\,du\\
&\ls\sum_{k=0}^\fz\frac{\|\xi\|}{(2^k\eta)^{Q+1}}\int_{2^k\eta/4(1+C_\rho)<\rho(g,u)\le2^{k+1}\eta/4(1+C_\rho)}|f(u)|\,du\\
&\ls\sum_{k=0}^\fz\frac{\|\xi\|}{2^k\eta}\frac{1}{(2^k\eta)^Q}\int_{B(g,\,2^{k+1}\eta/4(1+C_\rho))}|f(u)|\,du
\ls\frac{\|\xi\|}\eta \cm f(g).
\end{align*}
Then, by the boundedness of $\cm$ on $L_w^{p,\,\kappa} (\cc)$, we obtain
$$\|L_2\|_{L_w^{p,\,\kappa} (\cc)}\ls\frac{\|\xi\|}\eta\|f\|_{L_w^{p,\,\kappa} (\cc)}.$$
Consequently,
$[b,\,\mathcal C_\eta]\mathcal{F}$ satisfies condition (iii) of Lemma \ref{l-fre kol}.
Thus, $[b,\,\mathcal C_\eta]$ is a compact operator for any $b\in C^\fz_c(\cc)$.
This finishes the  proof of Theorem \ref{CZ com compact}(i).
\end{proof}

\subsection{Proof of Theorem \ref{CZ com compact}(ii)}\label{s3.2}

We begin with recalling an equivalent characterization of $\cmoc$ from
\cite[Theorem 4.4]{CDLW}.

\begin{lem}\label{l-cmo char}
Let $f\in\bmoc$.
Then $f\in\cmoc$ if and only if $f$ satisfies the following three conditions:
\begin{flalign*}\begin{split}
&{\rm (i)} \quad \quad\lim\limits_{a\rightarrow 0}\sup\limits_{r_B=a}M(f,B)=0;\\
&{\rm (ii)}  \quad\ \ \lim\limits_{a\rightarrow \infty}\sup\limits_{r_B=a}M(f,B)=0;\\
&{\rm (iii)}  \quad\ \lim\limits_{r\rightarrow \infty}\sup\limits_{B\subset \cc\setminus B(0,r)}M(f,B)=0,
\end{split}&
\end{flalign*}
where $r_B$ is the radius of the ball $B$.
\end{lem}

Next, we establish a lemma for the upper and the lower bounds of integrals of
$[b,\,\mathcal C]f_j$ on certain balls $B_j$ in $\cc$ for any $j\in\nn$.
It is easy to show that, for any $f\in L^1_{\rm loc}(\cc)$ and ball $B\subset \cc$,
\begin{equation}\label{equi osci}
M(f;B)\sim\frac1{|B|}\int_B\lf|f(g)-\az_B(f)\r|\,dg
\end{equation}
with the equivalent positive constants independent of $f$ and $B$.


\begin{lem}\label{l-cmo-contra}
Let $p\in(1,\infty)$, $\kappa\in(0,1)$ and $w\in A_p(\cc)$.
Suppose that $b\in{\rm BMO}(\cc)$ is a real-valued function with
$\|b\|_{{\rm BMO}(\cc)}=1$ and there exist $\dz\in(0, \fz)$ and a ball $B_0=B(g_0, r_0)\subset\cc$
with $r_0>0$,
such that
\begin{equation}\label{lower bdd osci}
M(b; B_0)>\dz.
\end{equation}
Then there exist a real-valued function $f_0\in L_w^{p,\,\kappa}(\cc)$,
positive constants $K_0$ large enough, $\wz C_0$, $\wz C_1$ and $\wz C_2$, which are independent of $g_0$ and $r_0$,
such that, for any  integer $k\ge K_0$,
$\|f_0\|_{L_w^{p,\,\kappa}(\cc)}\le \wz C_0$,
\begin{equation}\label{lower upper lpbdd riesz comm}
\int_{B_0^k}\lf|\lf[b, \mathcal C\r]f_0(g)\r|^pw(g)\,dg\geq\wz C_1
\frac{\dz^p}{A_{2}^{kpQ}}\lf[w\lf(B_0\r)\r]^{\kappa-1} w\lf(A_{2}^kB_0\r),
\end{equation}
where $B_0^k:=\widetilde{A_{2}^{k-1}B_0}$ is the ball associates with $A_{2}^{k-1}B_0$  in Lemma \ref{lb},
and
\begin{equation}\label{lower upper lpbdd riesz comm2}
\int_{A_{2}^{k+1}B_0\setminus A_{2}^k B_0}\lf|\lf[b, \mathcal C\r]f_0(g)\r|^pw(g)\,dg\le \wz C_2\frac1{A_{2}^{kpQ}}\lf[w\lf(B_0\r)\r]^{\kappa-1}w\lf(A_{2}^{k}B_0\r).
\end{equation}
\end{lem}
\color{black}

\begin{proof}
Set
$$B_{0, 1}=\{g\in B_0:\  b(g)>\az_{B_0}(b)\},\quad 
B_{0, 2}=\{g\in B_0:\  b(g)<\az_{B_0}(b)\}.$$
 We define function $f_0$ as follows:
\begin{equation*}
f^{(1)}_0:=\chi_{B_{0,1}}-\chi_{B_{0,2}},\quad f^{(2)}_0:=a_0\chi_{B_0}
\end{equation*}
and
$$f_0:=\lf[w\lf(B_0\r)\r]^{\frac{\kappa-1}p}\lf(f^{(1)}_0-f^{(2)}_0\r),$$
 where $B_0$ is as in the assumption of Lemma \ref{l-cmo-contra} and $a_0\in\rr$ is a
 constant such that
 \begin{equation}\label{fj proper-2}
 \int_{\cc} f_0(g)\,dg=0.
  \end{equation}
Then, by the definition of $a_0$, \eqref{median value-1} and \eqref{median value-2},
we have $|a_0|\le 1/2$, $\supp(f_0)\subset B_0$ and, for any $g\in B_0$,
\begin{equation}\label{fj proper-1}
 f_0(g)\lf(b(g)-\az_{B_0}(b)\r)\ge0.
 \end{equation}
Moreover, since $|a_0|\le1/2$, we can obtain that, for any $g\in (B_{0,\,1}\cup B_{0,\,2})$,
  \begin{equation}\label{fj proper-3}
   \lf|f_0(g)\r|\sim \lf[w\lf(B_0\r)\r]^{\frac{\kappa-1}p}
  \end{equation}
and hence
\begin{equation*}\begin{split}
\lf\|f_0\r\|_{L_w^{p,\,\kappa}(\cc)}&\ls\sup_{B\subset\cc}
\lf\{\frac{w(B\cap B_0)}{[w(B)]^\kappa}\r\}^{\frac1p}\lf[w\lf(B_0\r)\r]^{\frac{\kappa-1}p}\\
&\ls \sup_{B\subset\cc}\lf[w(B\cap B_0)\r]^{\frac{1-\kappa}p}
\lf[w\lf(B_0\r)\r]^{\frac{\kappa-1}p}\ls 1.
\end{split}\end{equation*}

Observe that, for any $k\in\nn$, we have
\begin{equation}\label{ijk set inclu}
A_{2}^{k-1}B_0\subset (A_{2}+1)B_0^k\subset A_{2}^{k+1}B_0
\end{equation}
and  hence
\begin{equation}\label{ijk meas}
w\lf(B_0^k\r)\sim w\lf(A_{2}^{k}B_0\r).
\end{equation}

Note that
\begin{equation}\label{com equiv}
\lf[b, \mathcal C\r]f=\lf[b-\az_{B_0}(b)\r]\mathcal C(f)-
  \mathcal C\lf(\lf[b-\az_{B_0}(b)\r]f\r).
\end{equation}
 By
Theorem B, \eqref{fj proper-2}, \eqref{fj proper-3}
and the fact that $\rho(g,g_0)\sim\rho(g,\xi)$ for any $g\in B_0^k$ with integer $k\geq2$
and $\xi\in B_0$, we have, for any $g\in B_0^k$,
\begin{align}\label{upper bdd riesz ope}
\lf|\lf[b(g)-\az_{B_0}(b)\r]\mathcal C(f_0)(g)\r|
&=\lf|b(g)-\az_{B_0}(b)\r|\lf|\int_{B_0}\lf[K(g, \xi)-
   K(g, g_0)\r]f_0(\xi)\,d\xi\r|\\
&\le\lf|b(g)-\az_{B_0}(b)\r|\int_{B_0}\lf|K(g, \xi)-
   K(g, g_0)\r|\lf|f_0(\xi)\r|\,d\xi\noz\\
&\ls \lf[w\lf(B_0\r)\r]^{\frac{\kappa-1}p}\lf|b(g)-\az_{B_0}(b)\r|\int_{B_0}
\frac{\rho(\xi,g_0)}{\rho(g,g_0)^{Q+1}}\,d\xi\noz\\
&\ls \lf[w\lf(B_0\r)\r]^{\frac{\kappa-1}p}r^{Q+1}_0\frac{|b(g)-\az_{B_0}(b)|}{\rho(g,g_0)^{Q+1}}\noz\\
&\ls \frac{[w(B_0)]^{\frac{\kappa-1}p}}{A_{2}^{k(Q+1)}}\lf|b(g)-\az_{B_0}(b)\r|.\noz
\end{align}

Since $\|b\|_{{\rm BMO}(\cc)}=1$, by {\eqref{bmoeq}}, for each $k\in \nn$ and ball $B\subset\cc$, 
we have
\begin{align}\label{b-bmo-bdd}
&\int_{A_{2}^{k+1}B}\lf|b(g)-\az_{B}(b)\r|^p\,dg\noz\\
&\ls \int_{A_{2}^{k+1}B}\lf|b(g)-\az_{A_{2}^{k+1}B}(b)\r|^p\,dg
+\lf|A_{2}^{k+1}B\r|\lf|\az_{A_{2}^{k+1}B}(b)-\az_{B}(b)\r|^p\\
&\ls k^p\lf|A_{2}^kB\r|,\noz
\end{align}
where the last inequality is due to the fact that
\begin{align*}
  \lf|\az_{A_{2}^{k+1}B}(b)-\az_{B}(b)\r|\ls
  \lf|\alpha_{A_{2}^{k+1}B}(b)-b_{A_{2}^{k+1}B}\r|
  +\lf|b_{A_{2}^{k+1}B}-b _{B}\r|
  +\lf|b_{B}-\alpha_{B}(b)\r|
  \ls k.
\end{align*}

Since $w\in A_p(\cc)$, there exists $\ez\in(0, \fz)$
such that the reverse H\"older inequality
$$\left[\frac1{|B|}\int_B w(g)^{1+\ez}\,dg\r]^{\frac1{1+\ez}}\ls \frac1{|B|}\int_B w(g)\,dg$$
holds for any ball $B\subset \cc$. Then by the H\"older inequality,
 \eqref{b-bmo-bdd}, \eqref{ijk set inclu} and \eqref{upper bdd riesz ope}, we can deduce that
 there exists a positive constant $\wz C_3$ such that, for any $k\in \nn$,
\begin{align}\label{upper bdd com}
&\int_{B_0^k}\lf|\lf[b(g)-\az_{B_0}(b)\r]\mathcal C(f_0)(g)\r|^pw(g)\,dg\\
&\quad\ls \frac{[w(B_0)]^{\kappa-1}}{A_{2}^{kp(Q+1)}}\int_{A_{2}^{k+1}B_0}\lf|b(g)-\az_{B_0}(b)\r|^pw(g)\,dg\noz\\
&\quad\ls \frac{[w(B_0)]^{\kappa-1}}{A_{2}^{kp(Q+1)}}\lf|A_{2}^{k}B_0\r|\lf[\frac1{|A_{2}^{k+1}B_0|}
\int_{A_{2}^{k+1}B_0}\lf|b(g)-\az_{B_0}(b)\r|^{p(1+\ez)'}\,dg\r]^{\frac1{(1+\ez)'}}\noz\\
&\quad\quad\times\lf[\frac1{|A_{2}^{k+1}B_0|}\int_{A_{2}^{k+1}B_0}w(g)^{1+\ez}\,dg\r]
  ^{\frac1{1+\ez}}\noz\\
&\quad\ls \frac{k^p}{A_{2}^{kp(Q+1)}}\lf[w\lf(B_0\r)\r]^{\kappa-1}w\lf(A_{2}^{k+1}B_0\r)\\
&\quad\leq \wz C_3\frac{k^p}{A_{2}^{kp(Q+1)}}\lf[w\lf(B_0\r)\r]^{\kappa-1}w\lf(A_{2}^{k}B_0\r)\noz.
\end{align}

By Lemma \ref{lb},  \eqref{fj proper-1}, \eqref{fj proper-3},
\eqref{equi osci} and \eqref{lower bdd osci}, for any $g\in B_0^k$, we have
\begin{align*}
\lf|\mathcal C\lf(\lf[b-\az_{B_0}(b)\r]f_0\r)(g)\r|
&=\lf|\int_{B_{j,\,1}\cup B_{j,\,2}}
K(g,\xi)[b(\xi)-\az_{B_0}(b)]f_0(\xi)d\xi\r|\\
&\gs\int_{B_{j,\,1}\cup B_{j,\,2}}
\frac{|[b(\xi)-\az_{B_0}(b)]f_0(\xi)|}{\rho(g,\xi)^{Q}}\,d\xi\\
&\gs\frac{1}{\rho(g,g_0)^{Q}}\lf[w\lf(B_0\r)\r]^{\frac{\kappa-1}p}
\int_{B_0}\lf|b(\xi)-\az_{B_0}(b)\r|\,d\xi\\
&\gs \frac{\dz}{A_{2}^{kQ}} \lf[w\lf(B_0\r)\r]^{\frac{\kappa-1}p}.
\end{align*}
Then together with \eqref{ijk meas}, we obtain that there exists a positive constant $\wz C_4$ such that
\begin{align}\label{low bdd com}
\quad\quad\int_{B_0^k}\lf|\mathcal C\lf(\lf[b-\az_{B_0}(b)\r]f_0\r)(g)\r|^pw(g)\,dg
&\gs \frac{\dz^p}{A_{2}^{kpQ}}\lf[w\lf(B_0\r)\r]^{\kappa-1} w\lf(B_0^k\r)\\
&\ge \wz C_4\frac{\dz^p}{A_{2}^{kpQ}}\lf[w\lf(B_0\r)\r]^{\kappa-1} w\lf(A_{2}^kB_0\r)\noz.
\end{align}

Take $K_0\in(0,\infty)$ large enough such that, for any integer $k\ge K_0$,
$$\wz C_4\frac{\dz^p}{2^{p-1}}-\wz C_3\frac{k^p}{A_{2}^{kp}}\ge \wz C_4\frac{\dz^p}{2^p}.$$
From this, \eqref{com equiv},  \eqref{upper bdd com} and \eqref{low bdd com},
we have
\begin{align*}
&\int_{B_0^k}\lf|[b,\,\mathcal C]f_0(z)\r|^pw(g)\,dg\noz\\
&\quad\ge\frac1{2^{p-1}}\int_{B_0^k}\lf|\mathcal C\lf(\lf[b-\az_{B_0}(b)\r]f_0\r)(g)\r|^pw(g)\,{dg}
-\int_{B_0^k}\lf|\lf[b(g)-\az_{B_0}(b)\r]\mathcal C(f_0)(g)\r|^pw(g)\,dg\noz\\
&\quad\ge\lf(\wz C_4\frac{\dz^p}{2^{p-1}}-\wz C_3\frac{k^p}{A_{2}^{kp}}\r)
\frac1{A_{2}^{kpQ}}\lf[w\lf(B_0\r)\r]^{\kappa-1} w\lf(A_{2}^kB_0\r)\\
&\ge \frac{\wz C_4}{2^p}
\frac{\dz^p}{A_{2}^{kpQ}}\lf[w\lf(B_0\r)\r]^{\kappa-1} w\lf(A_{2}^kB_0\r),
\end{align*}
which imply \eqref{lower upper lpbdd riesz comm}.

On the other hand, since $\supp(f_0)\subset B_0$, by Theorem B, \eqref{fj proper-3},
\eqref{equi osci} and $\|b\|_{{\rm BMO}(\cc)}=1$,
we can obtain that, for any $g\in A_{2}^{k+1}B_0\setminus A_{2}^{k}B_0$,
\begin{align*}
\lf|\mathcal C\lf(\lf[b-\az_{B_0}(b)\r]f_0\r)(g)\r|\ls \lf[w\lf(B_0\r)\r]^{\frac{\kappa-1}p}
\int_{B_0}\frac{|b(\xi)-\az_{B_0}(b)|}{\rho(g,\xi)^Q}\,d\xi
\ls \lf[w\lf(B_0\r)\r]^{\frac{\kappa-1}p}\frac1{A_{2}^{kQ}}.
\end{align*}
Therefore, by \eqref{upper bdd com}  with $B^k_0$ replaced by
$A_{2}^{k+1}B_0\setminus A_{2}^kB_0$, we can deduce that, for any integer $k\ge K_0$,
\begin{align*}
&\int_{A_{2}^{k+1}B_0\setminus A_{2}^kB_0}\lf|[b,\,\mathcal C]f_0(g)\r|^pw(g)\,dg\\
&\quad\ls\int_{A_{2}^{k+1}B_0\setminus A_{2}^kB_0}\lf|\mathcal C
   \lf(\lf[b-\az_{B_0}(b)\r]f_0\r)(g)\r|^pw(g)\,dg\\
  &\quad \quad+\int_{A_{2}^{k+1}B_0\setminus A_{2}^kB_0}\lf|\lf[b(g)-\az_{B_0}(b)\r]
   \mathcal C(f_0)(g)\r|^pw(g)\,dg\\
&\quad\ls \lf[w\lf(B_0\r)\r]^{\kappa-1}\frac1{A_{2}^{kpQ}}w\lf(A_{2}^{k+1}B_0\r)+\frac{k^p}{A_{2}^{kp(Q+1)}}
  \lf[w\lf(B_0\r)\r]^{\kappa-1}w\lf(A_{2}^{k}B_0\r)\\
&\quad\ls \lf[w\lf(B_0\r)\r]^{\kappa-1}\frac1{A_{2}^{kpQ}}w\lf(A_{2}^{k}B_0\r).
\end{align*}
This finishes the proof of Lemma \ref{l-cmo-contra}.
\end{proof}

\begin{rem}
The reality of $b$ is used in the construction of the sets $B_{0, 1}$, $B_{0, 2}$ and the important inequality \eqref {fj proper-1} in the above proof.  Therefore, we require $b$ to be real in Thereom \ref{CZ com bounded}-\ref{CZ com compact}, while such condition is also required in the main results of \cite{TYY,TYY2}. It is an interesting question to see whether Thereom \ref{CZ com bounded}-\ref{CZ com compact} in this article hold for quaternionic valued $b$.
\end{rem}

\color{black}


We also need the following technical result  
to handle the weighted estimate for the necessity of the compactness
of the commutators.

\begin{lem}\label{l-comp contra awy ori}
Let $1<p<\infty$, $0<\kappa<1$, $w\in A_p(\cc)$,
$b\in {\rm BMO}(\cc),\,\delta,\,K_0>0$.
{Assume that $B_j:=B(g_j, r_j)$, $j\in\nn$,  satisfies \eqref {lower bdd osci} and}
the following two conditions:
\begin{itemize}
  \item [{\rm(i)}] For any $ \ell,\,m\in\nn$ and $\ell\neq m$,
    \begin{equation}\label{I-j-pro}
    A_2C_1B_{\ell}\bigcap A_2C_1B_m=\emptyset.
    \end{equation}
    where $C_1:=A_{2}^{K_1}> C_2:=A_{2}^{K_0}$ for some $K_1\in\nn$ large enough.
  \item [{\rm(ii)}] $\{r_j\}_{j\in\nn}$ is either non-increasing or non-decreasing
    in $j$, or there exist positive constants $C_{\mathrm{min}}$
    and $C_{\mathrm{max}}$ such that, for any $j\in\nn$,
    $$C_{\mathrm{min}}\leq r_j\leq C_{\mathrm{max}}.$$
\end{itemize}
{Let $f_j$, $j\in\nn$, be the function constructed in  Lemma \ref{l-cmo-contra} for $B_0$ to be $B_j$.}
Then there exists a positive constant $C$ such that, for any $j,\,m\in\nn$,
$$\lf\|[b,\mathcal C]f_j- [b,\mathcal C]f_{j+m}\r\|_{L_w^{p,\,\kappa}(\cc)}\geq C.$$
\end{lem}

\begin{proof}
Without loss of generality, we may assume that $\|b\|_{{\rm BMO}(\cc)}\!=1$
and $\{r_j\}_{j\in\nn}$ is non-increasing.
{Let  $\wz C_1$, $\wz C_2$ be as in Lemma \ref{l-cmo-contra} such that \eqref {lower upper lpbdd riesz comm} and \eqref {lower upper lpbdd riesz comm2} hold for $B_j, f_j$ uniformly. }
associated with $\{B_j\}_{j\in\nn}$.
Recall that, for any $w\in A_p(\cc)$ with $p\in(1,\fz)$, there exists $p_0\in (1, p)$
such that $w\in A_{p_0}(\cc)$.
By \eqref{lower upper lpbdd riesz comm}, \eqref{ijk meas}, Lemma \ref{lemlw} with $w\in A_{p_0}(\cc)$, we find that, for any $j\in\nn$,
\begin{align}\label{riesz-comm-lower-bdd-on-A2Ij}
&\lf[\int_{C_1 B_j}\lf|[b,\mathcal C]f_j(g)\r|^pw(g)\,dg\r]
  ^{1/p}\lf[w\lf(C_1 B_j\r)\r]^{-\kappa/p}\\
&\quad\geq \lf[w\lf(C_1 B_j\r)\r]^{-\kappa/p}\lf\{\int_{B_{j}^{K_0-1}}
\lf|[b,\mathcal C]f_j(g)\r|^pw(g)\,dg\r\}^{1/p}\noz\\
&\quad\ge \lf[w\lf(C_1 B_j\r)\r]^{-\kappa/p}\lf\{\wz C_1 \delta^p
\frac{[w(B_{j})]^{\kappa -1}w(A_{2}^{{K_0-1}}B_{j})}{A_{2}^{p{(K_0-1)}Q}}\r\}^{1/p}\noz\\
&\quad\gtrsim \lf[w\lf(C_1 B_j\r)\r]^{-\kappa/p}\lf\{ \delta^p
\frac{[w(B_j)]^\kappa}{A_{2}^{Q(p-\sigma)(K_0-1)}}\r\}^{1/p}\noz\\
&\quad\ge C_3 C_1^{-\frac{Q\kappa}pp_0}\lf[w\lf(B_j\r)\r]^{-\kappa/p}\dz
 \lf[w\lf(B_{j}\r)\r]^{\kappa/p}= C_3\dz C_1^{-\frac{Q\kappa}pp_0}\noz
\end{align}
for some positive constant $C_3$ independent of $\delta$ and $C_1$.
We next prove that, for any $j,\,m\in\nn$,
\begin{equation}\label{riesz-comm-upper-bdd-on-A2Ij}
\lf[\int_{C_1 B_j}\lf|[b,\mathcal C]f_{j+m}(g)\r|^pw(g)\,dg\r]^{1/p}
\lf[w\lf(C_1 B_j\r)\r]^{-\kappa/p}\le \frac 12C_3\dz C_1^{-\frac{Q\kappa}pp_0}.
\end{equation}
Indeed, since $\supp (f_{j+m})\subset B_{j+m}$, from \eqref{fj proper-3}, \eqref{equi osci},
\eqref{I-j-pro} and $\|b\|_{{\rm BMO}(\cc)}=1$, it follows that, for any $g\in C_1B_j$,
\begin{align*}
\lf|\mathcal C\lf(\lf[b-\alpha_{B_{j+m}}(b)\r]f_{j+m}\r)(g)\r|
&\ls \lf[w\lf(B_{j+m}\r)\r]^{\frac{\kappa-1}p}\int_{B_{j+m}}|K(g,\xi)|
  \lf|b(g)-\alpha_{B_{j+m}}(b)\r|\,d\xi\\
&\ls \lf[w\lf(B_{j+m}\r)\r]^{\frac{\kappa-1}p}\frac{r_{j+m}^Q}{\rho(g_j,g_{j+m})^Q}
\end{align*}
and hence
\begin{align}\label{riesz-comm-upper-bdd-on-A2Ij-i}
&\lf\{\int_{C_1 B_j}\lf|\mathcal C\lf(\lf[b-\alpha_{B_{j+m}}(b)\r]f_{j+m}\r)(g)\r|^p
w(g)\,dg\r\}^{1/p}\lf[w\lf(C_1 B_j\r)\r]^{-\kappa/p}\\
&\quad\ls \lf[w\lf(B_{j+m}\r)\r]^{\frac{\kappa-1}p}\frac{r_{j+m}^Q}{\rho(g_j,g_{j+m})^Q}
\lf[w\lf(C_1 B_j\r)\r]^{\frac{1-\kappa}p}\noz\\
&\quad\ls \lf[w\lf(B_{j+m}\r)\r]^{\frac{\kappa-1}p}\frac{r_{j+m}^Q}{\rho(g_j,g_{j+m})^Q}
\lf[w\lf(\frac{\rho(g_j,g_{j+m})}{r_{j+m}} B_{j+m}\r)\r]^{\frac{1-\kappa}p}\noz\\
&\quad\ls \frac{r_{j+m}^Q}{\rho(g_j,g_{j+m})^Q}\lf(\frac{\rho(g_j,g_{j+m})}{r_{j+m}}\r)
  ^{Q\frac{1-\kappa}p p_0}
\sim \lf(\frac{\rho(g_j,g_{j+m})}{r_{j+m}}\r)^{-\frac{Q\kappa}{p}p_0+\frac{Qp_0}{p}-Q}\noz.
\end{align}

Moreover, from Theorem B and \eqref{fj proper-3},
we deduce that, for any $g\in C_1B_j$,
\begin{align}\label{riz-upper-bdd-case-iii}
\lf|\mathcal C(f_{j+m})(g)\r|&\le\int_{B_{j+m}}\lf|K(g,\xi)-
K(g,g_{j+m})\r|\lf|f_{j+m}(\xi)\r|\,d\xi\\
&\ls\int_{B_{j+m}}\frac{r_{j+m}}{\rho(g_j,g_{j+m})^{Q+1}}\lf|f_{j+m}(\xi)\r|\,d\xi\noz\\
&\ls \lf[w\lf(B_{j+m}\r)\r]^{\frac{\kappa-1}p}\frac{r_{j+m}^{Q+1}}{\rho(g_j,g_{j+m})^{Q+1}}\noz.
\end{align}
Then, by \eqref{riz-upper-bdd-case-iii}, the fact that
$\{r_j\}_{j\in\nn}$ is non-increasing in $j$,
the  H\"older and the reverse H\"older inequalities, we conclude that
\begin{align}\label{riesz-comm-upper-bdd-on-A2Ij-ii}
&\lf\{\int_{C_1 B_j}\lf|\lf[b(g)-\alpha_{B_{j+m}}(b)\r]\mathcal C(f_{j+m})(g)\r|^pw(g)\,dg\r\}^{1/p}\lf[w\lf(C_1 B_j\r)\r]^{-\kappa/p}\\
&\quad\ls \lf[w\lf(B_{j+m}\r)\r]^{\frac{\kappa-1}p}\frac{r_{j+m}^{Q+1}}{\rho(g_j,g_{j+m})^{Q+1}}
  \lf[w\lf(C_1 B_j\r)\r]^{-\kappa/p}\lf[\int_{C_1 B_j}\lf|b(g)-\alpha_{B_{j+m}}(b)\r|^p
  w(g)\,dg\r]^{1/p}\noz\\
&\quad\ls \lf[w\lf(B_{j+m}\r)\r]^{\frac{\kappa-1}p}\frac{r_{j+m}^{Q+1}}{\rho(g_j,g_{j+m})^{Q+1}}
  \lf[w\lf(C_1 B_j\r)\r]^{\frac{1-\kappa}p}\lf(\log\frac{\rho(g_j,g_{j+m})}{r_{j+m}}
+\log\frac{\rho(g_j,g_{j+m})}{r_{j}}\r)\noz\\
&\quad\ls \lf[w\lf(B_{j+m}\r)\r]^{\frac{\kappa-1}p}\frac{r_{j+m}^{Q+1}}{\rho(g_j,g_{j+m})^{Q+1}}
  \lf[w\lf(\frac{\rho(g_j,g_{j+m})}{r_{j+m}} B_{j+m}\r)\r]^{\frac{1-\kappa}p}
  \log\frac{\rho(g_j,g_{j+m})}{r_{j+m}}\noz\\
&\quad\ls \lf(\frac{\rho(g_j,g_{j+m})}{r_{j+m}}\r)^{-\frac{Q\kappa}{p}p_0+\frac{Qp_0}{p}-Q-1}
   \log\frac{\rho(g_j,g_{j+m})}{r_{j+m}}\noz.
\end{align}
Notice that, for $C_1$ large enough, by \eqref{I-j-pro},
we know that $\rho(g_j,g_{j+m})$ is also large enough and hence
\begin{align}\label{log decrease}
\lf(\frac{\rho(g_j,g_{j+m})}{r_{j+m}}\r)^{-1}\log\frac{\rho(g_j,g_{j+m})}{r_{j+m}}\lesssim1.
\end{align}
Therefore, from \eqref{riesz-comm-upper-bdd-on-A2Ij-i},
\eqref{riesz-comm-upper-bdd-on-A2Ij-ii}, \eqref{log decrease} and $p_0\in(1,p)$,
we deduce that, for $C_1$ large enough,
\begin{align*}
&\lf\{\int_{C_1 B_j}\lf|[b,\,\mathcal C](f_{j+m})(g)\r|^pw(g)\,dg\r\}^{1/p}
 \lf[w\lf(C_1 B_j\r)\r]^{-\kappa/p}\\
&\quad\le\lf\{\int_{C_1 B_j}\lf|\mathcal C\lf(\lf[b-\alpha_{B_{j+m}}(b)\r]f_{j+m}\r)(g)\r|^p
 w(g)\,dg\r\}^{1/p}\lf[w\lf(C_1 B_j\r)\r]^{-\kappa/p}\\
&\quad\quad+\lf\{\int_{C_1 B_j}\lf|\lf[b(g)-\alpha_{B_{j+m}}(b)\r]\mathcal C(f_{j+m})(g)\r|^p
 w(g)\,dg\r\}^{1/p}\lf[w\lf(C_1 B_j\r)\r]^{-\kappa/p}\\
&\quad\lesssim\lf(\frac{\rho(g_j,g_{j+m})}{r_{j+m}}\r)^{-\frac{Q\kappa}{p}p_0+\frac{Qp_0}{p}-Q}
\lf[1+\lf(\frac{\rho(g_j,g_{j+m})}{r_{j+m}}\r)^{-1}\log\frac{\rho(g_j,g_{j+m})}{r_{j+m}}\r]\\
&\quad\lesssim\lf(\frac{\rho(g_j,g_{j+m})}{r_{j+m}}\r)^{-\frac{Q\kappa}{p}p_0+\frac{Qp_0}{p}-Q}
       \lesssim\lf[\frac{A_2C_1(r_j+r_{j+m})}{r_{j+m}}\r]^{-\frac{Q\kappa}{p}p_0+\frac{Qp_0}{p}-Q}\\
 &\quad \lesssim C_1^{-\frac{Q\kappa}{p}p_0+\frac{Qp_0}{p}-Q}
       \le \frac12 C_3 \delta C_1^{-\frac{Q\kappa}{p}p_0}.
\end{align*}
This finishes the proof of \eqref{riesz-comm-upper-bdd-on-A2Ij}. By \eqref{riesz-comm-lower-bdd-on-A2Ij} and \eqref{riesz-comm-upper-bdd-on-A2Ij},
we know that, for any $j,\,m\in\nn$ and $C_1$ large enough,
\begin{align*}
&\lf\{\int_{C_1 B_j}\lf|[b,\,\mathcal C](f_j)(g)-[b,\,\mathcal C](f_{j+m})(g)\r|^p
  w(g)\,dg\r\}^{1/p}\lf[w\lf(C_1 B_j\r)\r]^{-\kappa/p}\\
&\quad\ge\lf\{\int_{C_1 B_j}\lf|[b,\,\mathcal C](f_j)(g)\r|^pw(g)\,dg\r\}^{1/p}
  \lf[w\lf(C_1 B_j\r)\r]^{-\kappa/p}\\
&\quad\quad-\lf\{\int_{C_1 B_j}\lf|[b,\,\mathcal C](f_{j+m})(g)\r|^p
  w(g)\,dg\r\}^{1/p}\lf[w\lf(C_1 B_j\r)\r]^{-\kappa/p}
  \ge\frac{1}{2}C_3\dz C_1^{-\frac{Q\kappa}pp_0}.
\end{align*}
This finishes the proof of Lemma \ref{l-comp contra awy ori}.
\end{proof}

\begin{proof}[Proof of Theorem \ref{CZ com compact}(ii)]
Without loss of generality, we may assume that $\|b\|_{{\rm BMO}(\cc)}=1$.
To show $b\in{\rm {VMO}}(\cc)$, noticing that $b\in {\rm BMO}(\cc)$ is a real-valued function,
we can use a contradiction argument via Lemmas \ref{l-cmo char}, \ref{l-cmo-contra}
and \ref{l-comp contra awy ori}.
Now observe that, if $b\notin {\rm {VMO}}(\cc)$, then $b$ does not satisfy at least one of
(i) through (iii) of Lemma \ref{l-cmo char}.
We show that $[b, \mathcal C]$ is not compact on $L_w^{p,\,\kappa}(\cc)$ in any of
the following three cases.

{\bf Case i)} $b$ does not satisfy Lemma \ref{l-cmo char}(i).
Then there exist $\dz\in(0, \fz)$ and a sequence
$$\{B^{(1)}_j\}_{j\in\nn}:=\{B(g_j^{(1)},r_j^{(1)})\}_{j\in\nn}$$ of balls in $\cc$
satisfying \eqref{lower bdd osci} and that $r_j^{(1)}\to0$ as $j\to\fz$.
We further consider the following two subcases.

{\bf Subcase (i)} There exists a positive constant $M$ such that $0\leq \|g^{(1)}_{j}\|<M$
for all $g^{(1)}_{j}$, $j\in\nn$. That is, $g^{(1)}_{j}\in B_0:=B(0, M)$, $\forall j\in\nn$.
Let  $\{f_j\}_{j\in\nn}$ be associated with $\{B_j\}_{j\in\nn}$,
$\wz C_1$, $\wz C_2$, $K_0$ and $C_2$ be as in Lemmas \ref{l-cmo-contra} and
\ref{l-comp contra awy ori}.
Let $p_0\in(1,p)$ be such that $w\in A_{p_0}(\cc)$ and $C_4:=A_2^{K_2}> C_2=A_2^{K_0}$
for $K_2\in\nn$ large enough such that
\begin{align}\label{C5}
C_5:=\wz C_1 \hat C_{2}\delta^p A_{2}^{QK_0(\sigma-p)}>
2\frac{\wz C_2}{1-A_{2}^{Q(p_0-p)}}\frac{\hat C}{A_{2}^{QK_2(p-p_0)}},
\end{align}
where $\hat C_2$ and $\hat C$ are as in Lemma \ref{lemlw}.
Since $|r_j^{(1)}|\to 0$ as $j\to\fz$ and $\{g^{(1)}_j\}_{j\in\nn}\subset B_0$,
we may choose a subsequence 

of $\{B^{(1)}_j\}_{j\in\nn}$, for simplicity, we still denote it by $\{B^{(1)}_j\}_{j\in\nn}$,
\color{black}
 such that, for any $j\in\nn$,
\begin{equation}\label{descreasing interval}
\frac{|B_{j+1}^{(1)}|}{|B_{j}^{(1)}|}<\frac1{C_4^Q}\quad{\rm and\quad} w\big(B_{j+1}^{(1)}\big)\le w\big(B_{j}^{(1)}\big).
\end{equation}
For fixed $\ell,\ m\in \mathbb N$, let
$$\mathcal J:=C_4B^{(1)}_{\ell}\setminus C_2B^{(1)}_{\ell},
\quad\mathcal J_1:=\mathcal J\setminus C_4B^{(1)}_{\ell+m}
\quad\textrm{and}\quad\mathcal J_2:=\cc\setminus C_4B^{(1)}_{\ell+m}.$$
Notice that
$$\mathcal J_1\subset \lf[\lf(C_4B^{(1)}_{\ell}\r)\cap\mathcal J_2\r]
\quad{\rm and}\quad \mathcal J_1=\mathcal J\cap \mathcal J_2.$$
 We then have
\begin{align}\label{low lpbdd comparing com}
&\lf\{\int_{C_4B_{\ell}^{(1)}}\lf|\lf[b,\mathcal C\r](f_{\ell})(g)
-\lf[b,\mathcal C\r](f_{\ell+m})(g)\r|^pw(g)\,dg\r\}^{1/p}\\
&\quad\ge\lf\{\int_{\mathcal J_1}\lf|\lf[b,\mathcal C\r](f_{\ell})(g)
-\lf[b,\mathcal C\r](f_{\ell+m})(g)\r|^pw(g)\,dg\r\}^{1/p}\noz\\
&\quad\ge\lf\{\int_{\mathcal J_1}\lf|\lf[b,\mathcal C\r](f_{\ell})(g)\r|^pw(g)\,dg\r\}^{1/p}
-\lf\{\int_{\mathcal J_2}\lf|\lf[b,\mathcal C\r](f_{\ell+m})(g)\r|^pw(g)\,dg\r\}^{1/p}\noz\\
&\quad=\lf\{\int_{\mathcal J\cap \mathcal J_2}
\lf|\lf[b,\mathcal C\r](f_{\ell})(g)\r|^pw(g)\,dg\r\}^{1/p}
-\lf\{\int_{\mathcal J_2}\lf|\lf[b,\mathcal C\r](f_{\ell+m})(g)\r|^pw(g)\,dg\r\}^{1/p}\noz\\
&\quad=:{\rm F_1}-{\rm F_2}.\noz
\end{align}

We first consider the term ${\rm F_1}$. Assume that
$E_{{\ell+m}}:=\mathcal J\setminus\mathcal J_2\not=\emptyset$.
Then $E_{j_\ell}\subset C_4B^{(1)}_{\ell+m}$.
Thus, by \eqref{descreasing interval}, we have
\begin{align}\label{ee1}
|E_{\ell}|\le C_4^Q\lf|B^{(1)}_{\ell+m}\r|<\lf|B^{(1)}_{j_{\ell}}\r|.
\end{align}
Now let
$$B^{(1)}_{\ell,\,k}:=\widetilde{A_{2}^{k-1}B^{(1)}_{\ell}},$$
be the ball associates with $A_{2}^{k-1}B^{(1)}_{\ell}$  in Lemma \ref{lb}.
Then, from \eqref{ee1}, we have
$$\lf|B^{(1)}_{\ell,\,k}\r|=A_{2}^{Q(k-1)}\left|B^{(1)}_{\ell}\r|> |E_{\ell}|. $$
By this, we further know that there exist finite
$\{B^{(1)}_{\ell,\,k}\}_{k=K_0}^{K_2-2}$ intersecting $E_{\ell}$.
By \eqref{lower upper lpbdd riesz comm} and Lemma \ref{lemlw}, 
we conclude that
\begin{align}\label{F_1^p}
{\rm F}_1^p&\ge\sum_{k=K_0,\,B^{(1)}_{\ell,\,k}\cap E_{\ell}=\emptyset}^{K_2-2}
\int_{B^{(1)}_{\ell,\,k}}\lf|[b,\mathcal C](f_{\ell})(g)\r|^pw(g)\,dg\\
&\ge\wz C_1\dz^p\sum_{k=K_0,\,B^{(1)}_{\ell,\,k}\cap E_{\ell}=\emptyset}^{K_2-2}
\frac{[w(B^{(1)}_{\ell})]^{\kappa -1}w(A_{2}^{k}B^{(1)}_{\ell})}{A_{2}^{Qkp}}\noz\\
&\ge\wz C_1 \hat C_2\dz^p\sum_{k=K_0,\,B^{(1)}_{\ell,\,k}\cap E_{\ell}=\emptyset}^{K_2-2}\frac{[w(B^{(1)}_{\ell})]^{\kappa}}{A_{2}^{Qk(p-\sigma)}}\noz\\
&\ge \wz C_1 \hat C_2\dz^pA_{2}^{QK_0(\sigma-p)}\lf[w\lf(B^{(1)}_{\ell}\r)\r]^{\kappa}
=C_5\lf[w\lf(B^{(1)}_{\ell}\r)\r]^{\kappa}.\noz
\end{align}
If  $E_{\ell}:=\mathcal J\setminus\mathcal J_2=\emptyset$,
the inequality above is still true.

Moreover, from \eqref{lower upper lpbdd riesz comm2}, Lemma \ref{lemlw},
\eqref{C5} and \eqref{descreasing interval}, we deduce that
\begin{align}\label{F_2^p}
\quad{\rm F}^p_2&\le\sum_{k=K_2}^{\fz}
\int_{A_{2}^{k+1}B_{\ell+m}^{(1)} \setminus A_{2}^{k}B_{\ell+m}^{(1)}}
\lf|[b,\mathcal C](f_{\ell+m})(g)\r|^pw(g)\,dg\noz\\
&\le\wz C_2\sum_{k=K_2}^\fz
\frac{[w(B^{(1)}_{\ell+m})]^{\kappa -1}w(A_{2}^{k}B^{(1)}_{\ell+m})}{A_{2}^{Qkp}}\\
&\le\wz C_2\sum_{k=K_2}^\fz\frac{\hat C}{A_{2}^{Qk(p-p_0)}}
  \lf[w\lf(B^{(1)}_{\ell+m}\r)\r]^{\kappa}\noz\\
&\le \frac{\wz C_2}{1-A_{2}^{Q(p_0-p)}}\frac{\hat C}{A_{2}^{QK_2(p-p_0)}}
  \lf[w\lf(B^{(1)}_{\ell+m}\r)\r]^{\kappa}\noz\\
&<\frac {C_5}2\lf[w\lf(B^{(1)}_{\ell+m}\r)\r]^{\kappa}
 \le\frac{C_5}2\lf[w\lf(B^{(1)}_{\ell}\r)\r]^{\kappa}.\noz
\end{align}
By \eqref{low lpbdd comparing com}, \eqref{F_1^p} and \eqref{F_2^p}, we obtain
\begin{align*}
&\lf\{\int_{C_4B_{\ell}^{(1)}}\lf|[b,\mathcal C](f_{\ell})(g)-
  [b,\mathcal C](f_{\ell+m})(g)\r|^pw(g)\,dg\r\}^{1/p}\noz\\
&\quad\ge C_5^{1/p}\lf[w\lf(B^{(1)}_{\ell}\r)\r]^{\kappa/p}-
  \lf(\frac{C_5}{2}\r)^{1/p}\lf[w\lf(B^{(1)}_{\ell}\r)\r]^{\kappa/p}
  \gs \lf[w\lf(B^{(1)}_{\ell}\r)\r]^{\kappa/p}.
\end{align*}
Thus, $\{[b,\,\mathcal C]f_j\}_{j\in\nn}$ is not relatively compact in $L_w^{p,\,\kappa}(\cc)$,
which implies that
$[b,\,\mathcal C]$ is not compact on $L_w^{p,\,\kappa}(\cc)$. Therefore,
$b$ satisfies condition (i) of Lemma \ref{l-cmo char}.

{\bf Subcase (ii)} 
There exists a subsequence of
of $\{B^{(1)}_j\}_{j\in\nn}=\{B(g_j^{(1)},r_j^{(1)})\}_{j\in\nn}$, for simplicity, we still denote it by $\{B^{(1)}_j\}_{j\in\nn}$,  such that $|g^{(1)}_{j}|\to\infty$ as $j\to\infty$.
In this subcase, by $|B_{j}^{(1)}|\to 0$ as $j\to\infty$,
we can take a mutually disjoint subsequence of $\{B^{(1)}_{j}\}_{j\in\nn}$,
still denoted by $\{B^{(1)}_{j}\}_{j\in\nn}$, satisfying \eqref{I-j-pro} as well.
\color{black}
This, via Lemma \ref{l-comp contra awy ori}, implies that $[b, \mathcal C]$ is not compact on $L_w^{p,\,\kappa}(\cc)$, which is a contradiction to our assumption.
Thus, $b$ satisfies condition (i) of Lemma \ref{l-cmo char}.

{\bf Case ii)} If $b$ does not satisfy condition (ii) of Lemma \ref{l-cmo char}.
In this case, there exist $\delta\in(0, \fz)$ and a sequence
$\{B^{(2)}_j\}_{j\in\nn}$ of balls in $\cc$ satisfying \eqref{lower bdd osci}
and that $|r_{B^{(2)}_j}|\rightarrow\infty$ as $j\rightarrow\infty$.
We further consider the following two subcases as well.

{\bf Subcase (i)} 

There exists an infinite subsequence 
of $\{B^{(2)}_j\}_{j\in\nn}$, for simplicity, we still denote it by $\{B^{(2)}_j\}_{j\in\nn}$, 
\color{black}
and a point $g_0\in\cc$ such that, for any $j\in\nn$,
$g_0\in A_{2}C_1B^{(2)}_{j}$. Since  $|r_{B^{(2)}_{j}}|\rightarrow\infty$ as $j\rightarrow\infty$, it follows that there exists a subsequence, { still denoted by}
$\{B_{j}^{(2)}\}_{j\in\nn}$, such that, for any $j\in\nn$,
\begin{equation}\label{increasing interval}
\frac{|B_{j}^{(2)}|}{|B_{j+1}^{(2)}|}<\frac1{C_4^Q}.
\end{equation}
Observe that
$2A_2C_1B_{j}^{(2)}\subset 2A_2C_1B_{j+1}^{(2)}$ for any $j_\ell\in \nn$ and hence
\begin{align}\label{increasing interval weight}
w\lf(2A_2C_1B_ {j+1} ^{(2)}\r)\geq w\lf(2A_2C_1B_{j}^{(2)}\r),
\quad M\lf(b;2A_2C_1B_{j}\r)>\frac{\delta}{8A_{2}^2C_1^2}.
\end{align}
We can use a similar method as that used in Subcase (i) of Case i) and
redefine our sets in a reversed order. That is, for any fixed $\ell,\,k\in\nn$, let
\begin{equation*}\begin{split}
\widetilde{\mathcal J}:&=2A_2C_4C_1B_{\ell+k}^{(2)}\setminus 2A_2C_2C_1B_{\ell+k}^{(2)},\\
\widetilde{\mathcal J_1}:&=\widetilde{\mathcal J}\setminus 2A_2C_4C_1B_{\ell}^{(2)},\\
\widetilde{\mathcal J_2}:&=\cc\setminus 2A_2C_4C_1B_{\ell}^{(2)}.
\end{split}\end{equation*}
As in Case i), by Lemma \ref{l-cmo-contra}, \eqref{increasing interval}
and \eqref{increasing interval weight},
we conclude that the commutator $[b, \mathcal C]$ is not compact on $L_w^{p,\,\kappa}(\cc)$.
This contradiction implies that $b$ satisfies condition (ii) of Lemma \ref{l-cmo char}.

{\bf Subcase (ii)} For any $A_2\in\cc$, the number of $\{A_2C_1B^{(2)}_j\}_{j\in\nn}$
containing {$g_0$} is finite. In this subcase, for each ball
$B^{(2)}_{j_0}\in \{B^{(2)}_j\}_{j\in\nn}$,
the number of $\{A_2C_1B^{(2)}_j\}_{j\in\nn}$ intersecting $A_2C_1B^{(2)}_{j_0}$ is finite.
Then we take a mutually disjoint subsequence
of $\{B^{(2)}_{j}\}_{j\in\nn}$ satisfying
\eqref{lower bdd osci} and \eqref{I-j-pro}. From Lemma \ref{l-comp contra awy ori},
we deduce that $[b, \mathcal C]$ is not compact on $L_w^{p,\,\kappa}(\cc)$.
Thus, $b$ satisfies condition (ii) of Lemma \ref{l-cmo char}.

{\bf Case iii)} Condition (iii) of Lemma \ref{l-cmo char} does not hold for $b$.
Then there exists $\delta>0$ such that for any $r>0$, there exists $B\subset \cc\setminus B(0,r)$ with $M(b,B)>\delta$. 
As in \cite[p. 1661]{CDLW}, for the $\delta$ above, there exists a sequence $\{B^{(3)}_j\}_j$ of balls such that for any $j$,
\begin{align}\label{mbb}
M(b, B^{(3)}_j)>\delta,
\end{align}
and for any $i\neq m$,
\begin{align}\label{bb}
\gamma_1B^{(3)}_i\cap \gamma_1B^{(3)}_m=\emptyset,
\end{align}
for sufficiently large $\gamma_1$.


Since, by  Case i) and ii), $\{B^{(3)}_j\}_{j\in\nn}$ satisfies the conditions (i)
and (ii) of Lemma \ref{l-cmo char}, it follows that there exist positive constants
$C_{\mathrm{min}}$ and $C_{\mathrm{max}}$ such that
$$C_{\mathrm{min}}\leq r_j\leq C_{\mathrm{max}}, \quad \forall j\in\nn.$$
By this and Lemma \ref{l-comp contra awy ori}, we conclude that, if $[b, \mathcal C]$
is compact on $L_w^{p,\,\kappa}(\cc)$, then $b$  also  satisfies condition (iii) of
Lemma \ref{l-cmo char}. This finishes the proof of Theorem \ref{CZ com compact}(ii)
and hence of Theorem \ref{CZ com compact}.
\end{proof}

\bigskip
{\bf Acknowledgement:} 
The authors would like to thank Dr. Ji Li for his helpful advice and comments.
This work was supported by Natural Science Foundation
of China (Grant Nos. 11671185, 11701250 and 11771195) and Natural Science Foundation of Shandong
Province (Grant Nos. ZR2018LA002 and ZR2019YQ04).

\bibliographystyle{amsplain}

\end{document}